\documentclass[11pt]{amsart}
\usepackage{amscd,amsxtra,amssymb,mathrsfs, bbm}
\usepackage{stmaryrd}
\usepackage{amscd,amsxtra,amssymb, bbm, url}
\usepackage[new]{old-arrows}
\usepackage{enumitem}
\usepackage{hyperref}

\usepackage[paper=a4paper,left=30mm,right=30mm,top=30mm,bottom=30mm]{geometry}

\newtheorem{theorem}{Theorem}[section]
\newtheorem{corollary}[theorem]{Corollary}
\newtheorem{lemma}[theorem]{Lemma}
\newtheorem{proposition}[theorem]{Proposition}

\theoremstyle{definition}
\newtheorem{definition}[theorem]{Definition}
\newtheorem{remark}[theorem]{Remark}
\newtheorem{example}[theorem]{Example}
\newtheorem{defProp}[theorem]{Def./Proposition}

\numberwithin{equation}{section}

\DeclareMathOperator{\add}{\mathsf{add}}

\DeclareMathOperator{\rad}{\mathsf{rad}}

\DeclareMathOperator{\rk}{\mathrm{rk}}

\DeclareMathOperator{\cok}{\mathrm{cok}}
\renewcommand{\ker}{\mathsf{ker}}
\newcommand{\im}{\mathrm{im}}

\newcommand{\can}{\mathsf{can}}
\newcommand{\incl}{\mathsf{incl}}
\newcommand{\op}{\mathsf{op}}

\newcommand{\id}{\mathrm{id}}

\DeclareMathOperator{\ev}{\mathsf{ev}}

\DeclareMathOperator{\Coh}{\mathsf{Coh}}

\DeclareMathOperator{\Tor}{\mathsf{Tor}}

\newcommand{\lmod}{\mathsf{-mod}}
\newcommand{\rmod}{\mathsf{mod-}}

\DeclareMathOperator{\lcm}{\mathrm{lcm}}
\DeclareMathOperator{\Gen}{\mathsf{Gen}}

\DeclareMathOperator{\Hom}{\mathsf{Hom}}

\DeclareMathOperator{\Ext}{\mathsf{Ext}}

\DeclareMathOperator{\pdim}{\mathsf{pdim}}
\DeclareMathOperator{\tp}{\mathsf{top}}

\DeclareMathOperator{\End}{\mathsf{End}}

\def\bop{\bigoplus}

\newcommand{\CC}{\mathbb{C}}
\newcommand{\RR}{\mathbb{R}}

\newcommand{\ZZ}{\mathbb{Z}}

\def\bop{\bigoplus}

\newcommand{\XX}{\mathbb{X}}

\renewcommand{\SS}{\mathbb{S}}

\newcommand{\PP}{\mathbbm{P}}

\newcommand{\kC}{\mathcal{C}}
\newcommand{\kD}{\mathcal{D}}

\newcommand{\kH}{\mathcal{H}}
\newcommand{\kI}{\mathcal{I}}

\newcommand{\kP}{\mathcal{P}}

\newcommand{\kR}{\mathcal{R}}
\newcommand{\kK}{\mathcal{K}}

\newcommand{\kT}{\mathcal{T}}

\newcommand{\kY}{\mathcal{Y}}

\def\sD{\mathsf D}

\def\sH{\mathsf H}

\def\sK{\mathsf K}

\def\kron#1#2{\xymatrix@C=2em{{#1}\ar@/^3pt/[r]\ar@/_3pt/[r]&{#2}}}

\usepackage{tikz}
\usetikzlibrary{calc, arrows, positioning, shapes, fit, matrix, decorations}
\usetikzlibrary{decorations.shapes, decorations.pathreplacing}

\tikzset{
  decorate with/.style={decorate,decoration={shape backgrounds,shape=#1,shape size=1.5mm}},
   deco/.style={decorate with=dart},
   ordi/.style={draw,-stealth,  thick},
   conj/.style={dashed, draw, thick},
   ve/.style={circle, draw, thick, fill=blue!20, inner sep=1pt, outer sep=2pt, minimum size=7pt},
    dot/.style={fill=blue!10,circle,draw, inner sep=1pt, minimum size=5pt},
  dv/.style={star,star points=5,
star point ratio=2, draw, thick, fill=green!20, inner sep=1pt,outer sep=2pt,minimum size=7pt}
}

\tikzset{
    tbl5 nodes/.style={
        rectangle,
        execute at begin node=$,
       execute at end node=$,
       fill=blue!5,
        align=center,
        text depth=0.5ex,
        text height=2ex,
        inner xsep=0pt,
        outer sep=0pt,
           },
    tbl5/.style={
        matrix of nodes,
        row sep=-\pgflinewidth,
        column sep=-\pgflinewidth,
        nodes={
            tbl5 nodes
        },
        execute at empty cell={\node[draw=none]{};}
    }
  }

\input xy
\xyoption{all}

\title[Coxeter-Dynkin algebras of canonical type]{Coxeter-Dynkin algebras of canonical type}

\author{Daniel Perniok}
\address{
Paderborn University
}
\email{dperniok@math.uni-paderborn.de}

\begin{document}

\begin{abstract}
We propose a definition of Coxeter-Dynkin algebras of canonical type
generalising the definition as a path algebra of a quiver as it is given for instance in  \cite{LenzingdelaPenaExtendedCanonical}. Moreover, we construct two tilting objects over the squid algebra \cite{Ringel} -- one via generalised APR-tilting and one via one-point-extensions and reflection functors -- and identify their endomorphism algebras with the Coxeter-Dynkin algebra. This shows that our definition gives another representative in the derived equivalence class of the squid algebra, and hence of the corresponding canonical algebra.
Finally, we have a closer look at the Grothendieck group and the Euler form which illustrates the connection to Saito's classification of marked extended affine root systems \cite{SaitoI}. On the other hand, this enables us to prove that in the domestic case Coxeter-Dynkin algebras are of finite representation type.
\end{abstract}

\maketitle

\tableofcontents

\section*{Introduction}
\renewcommand{\thefootnote}{}
\footnotetext{Date: August 27, 2026.}
\renewcommand{\thefootnote}{\arabic{footnote}}
It is well-known that indecomposable representations of a tame hereditary algebra split into three classes: the preprojective component, the preinjective component and the tubes (see e.\,g.~\cite{RingelTameAlgebras} or \cite{SimsonSkowronski2}). Whereas the preprojective and the preinjective representations can be controlled on the level of dimension vectors, the representations that sit in the tubes are more difficult to understand. This led to the construction and study of more algebras that admit so-called \emph{stable separating tubular families}. In particular, Ringel constructed \emph{canonical algebras} from quivers with relations as examples of such algebras \cite{RingelTameAlgebras}.
A recurring question in this context concerns the parametrisation of the tubes. While in general this might be an arbitrarily hard number-theoretic problem, over an algebraically closed field the answer is known. The closed points of the projective line $\PP^1_k$ serve as a parameter set, which hints at a hidden geometry behind the tubular family, even in the general case. Over an algebraically closed field, this led to the definition of \emph{weighted projective lines} \cite{GeigleLenzing}. Geigle and Lenzing constructed a derived equivalence between the category $\Coh(\XX)$ of coherent sheaves on a weighted projective line $\XX$ and the module category of a canonical algebra in the sense of \cite{RingelTameAlgebras}.
\\
In order to develop the theory without assumptions on the base field, Ringel gave a sophisticated definition of canonical algebras generalising the previous one \cite{Ringel}. He also defined another algebra, called \emph{squid algebra}, over which he constructed a tilting module whose endomorphism algebra is a canonical algebra. Building on this, Lenzing and de la Pe$\tilde{\text{n}}$a achieved an alternative characterisation of \emph{concealed-canonical} algebras (i.\,e.~those algebras which are derived equivalent to a canonical algebra via a special tilting module). They proved that these are precisely the algebras which admit a \emph{sincere separating tubular family} \cite{LenzingdelaPena}. Furthermore, they showed that this condition is equivalent to the existence of a certain hereditary category $\kH$ which -- over an algebraically closed field -- can be identified with $\Coh(\XX)$ in the sense of \cite{GeigleLenzing}.
\\
An axiomatic characterisation of these categories can be found in \cite{LenzingFDAlgebrasAndSingularityTheory}, where the concept of an \emph{exceptional curve} over $k$ is introduced via a list of conditions that an associated hereditary $k$-category $\kH=\Coh(\XX)$ is supposed to satisfy. Further investigations based on this approach were performed by Kussin, including several phenomena which are only visible over non-algebraically closed fields \cite{KussinMemoirs}.
\\
In a recent work, Burban provides a geometric description of the category $\kH$ by constructing -- up to Morita-equivalence -- a ringed space $\XX$, called \emph{exceptional hereditary curve}, which is the desired generalisation of Geigle-Lenzing's weighted projective lines \cite{BurbanExceptionalCurves}. He proves that $\Coh(\XX)$ admits a tilting object whose endomorphism algebra is a squid algebra which in turn is derived equivalent to the canonical algebra by \cite{Ringel}.
\\
Our goal is to look at yet another representative in the derived equivalence class of the canonical algebra and the squid algebra, the so-called \emph{Coxeter-Dynkin algebra of canonical type}. In \cite[§3.9]{LenzingdelaPenaExtendedCanonical} one can find a definition of Coxeter-Dynkin algebras as path algebras of quivers modulo relations, motivated by connections to singularity theory.
Moreover, related algebras and diagrams play a role in recent works on cluster categories coming from weighted projective lines \cite{FuGengClusterCategory} and cluster modular groups of cluster algebras \cite{GreenbergKaufman}.
A generalisation of the definition to all cases which can occur over non-algebraically closed fields is the first goal of this work. Moreover, we establish in two different ways that the Coxeter-Dynkin algebra $B$ is derived equivalent to the squid algebra $A$ and hence to the canonical algebra $C$. Firstly, we use generalised APR-tilting with the advantage that we obtain a classical tilting module over $A$ whose endomorphism algebra is isomorphic to $B$. Secondly, we use reflection functors for hereditary algebras and one-point-extensions to construct a tilting complex over $A$. The second approach has the advantage that it can offer a more intuitive and constructive viewpoint towards the definition of Coxeter-Dynkin algebras. Subsequently, we apply the idea of reflection functors and one-point-extensions to canonical algebras to give a detailed proof of Ringel's tilting theorem for the squid algebra and the canonical algebra and offer a new perspective on the definition of canonical algebras.
\\
The previously described equivalences are summarised in the following diagram:
\begin{align*}
\xymatrix
{
&&& \kD^b(\Coh(\XX)) \ar[d]^{\cong}_{\text{\cite{BurbanExceptionalCurves}}}
\ar@/_1pc/[dlll]^\cong_{\text{\cite{GeigleLenzing}},~ k=\bar{k}}
\ar@/^1pc/[drrr]_\cong^{\text{\quad\cite{LenzingdelaPenaExtendedCanonical}},~ k=\bar{k}}
&&& \\
\kD^b(\rmod C) &&& 
\kD^b(\rmod A) \ar[rrr]^\cong_{\text{Thm.~\ref{thm:tiltSquidToOctopus}~+~\ref{thm:tiltSquidToOctopus1ptext}}} \ar[lll]_\cong^{\text{\cite{Ringel}}} &&&
\kD^b(\rmod B)
}
\end{align*}
As a next step, we will discover that Coxeter-Dynkin algebras of canonical type illustrate the connection between the theory of canonical algebras and exceptional curves on one side and geometric group theory on the other side. It turns out that Coxeter-Dynkin algebras are closely related to the diagrams that Saito used to classify extended affine root systems of codimension $1$ and their associated elliptic Weyl groups \cite{SaitoI}. These connections recently have been studied in \cite{NonCrossingPartitions}. There, a poset isomorphism is established between the non-crossing partitions in the reflection group and the exceptional subcategories in the category of coherent sheaves over an exceptional hereditary curve.
\\
The final goal of this paper is to show that in the domestic case Coxeter-Dynkin algebras are tilted algebras of finite representation type, which is a new result even in the case of quivers with relations.
\\[+4pt]
\emph{Outline.} In Section \ref{section:squidAndCDalgebra}, we describe the initial data from which we start to define the algebras $A$, $B$ and $C$ consisting of a minimal model, exceptional points and weights. Then, we recall the definition of a squid algebra and prove the equivalence of several conditions under which we can define a reasonable notion of Coxeter-Dynkin algebra of canonical type. These equivalent conditions are closely related to the tilting statements which follow in sections \ref{section:squidToCDviageneralisedAPR} and \ref{section:CDalgebraViaOnePointExtension}. While the first tilting object in Section \ref{section:squidToCDviageneralisedAPR} arises from generalised APR-tilting, the second one in Section \ref{section:CDalgebraViaOnePointExtension} is based on a theorem of Barot and Lenzing concerning derived equivalences and one-point-extensions. In Section \ref{section:canonicalAlgebras}, we recall the definition of canonical algebras from \cite{Ringel} and reprove his tilting result about the squid and the canonical algebra from a different viewpoint. Section \ref{section:SymbolWeylGroup} is dedicated to the connection to Lenzing's study of canonical bilinear lattices \cite{LenzingKTheoreticStudy} and geometric group theory. We introduce the \emph{symbol} of a canonical algebra, explain how to extract it from the input data and how the corresponding trichotomy of canonical algebras looks like. Also this section serves as a preparation for Section \ref{section:RepresentationType} where we prove finite representation type of $B$ in the domestic case.  
\\[+4pt]
\emph{Notation.} Let $k$ be an arbitrary field. If $R$ is a finite-dimensional $k$-algebra, then we write $\rmod R$ (resp.~$R\lmod$)
for the category of finitely generated  
right (resp.~left) $R$-modules. We will frequently use notions of $k$-species and their tensor algebras (or as a special case quivers and their path algebras) and follow the convention of \cite{DlabRingel}. In particular, in the path algebra of a quiver, arrows are composed from left to right (hence opposite to the way how we compose functions). We write $\kD^b(\rmod R)$ for the bounded derived category of $\rmod R$ and $\kK^b(\mathsf{proj-}R)$ for the homotopy category of bounded complexes of projective right $R$-modules. Moreover, $\sD(-)=\Hom_k(-,k)$ is our notation for the dual with respect to the base field $k$.
\\[+4pt]
\emph{Acknowledgement.} 
This work was funded by the Deutsche Forschungsgemeinschaft (DFG, German Research Foundation) – Project-ID 491392403 – TRR 358.
I would like to thank Igor Burban and Charly Schwabe for many fruitful discussions and helpful comments on this work.
Moreover I am grateful to the referee for the thorough proofreading and valuable suggestions.

\section{Squid algebras and Coxeter-Dynkin algebras of canonical type}\label{section:squidAndCDalgebra}

Let $k$ be an arbitrary field and $F,G$ finite-dimensional division algebras over $k$. Let $M={}_FM_G$ be a finite-dimensional $F$-$G$-bimodule such that $\lambda m= m\lambda$ for all $\lambda\in k, m\in M$. Then
\[
\Lambda = \begin{pmatrix}
F & {}_FM_G \\ 0 & G
\end{pmatrix}
\]
is a finite-dimensional hereditary $k$-algebra. We can assume without loss of generality that $k$ is the center of $\Lambda$. The bimodule $M$ is called \emph{tame}, if $\Lambda$ is a tame hereditary algebra, or equivalently if
\[
\dim{}_FM\cdot\dim M_G = 4.
\]
From now on we assume that $M$ is a tame bimodule. In this case the category of finite-dimensional right $\Lambda$-modules decomposes as
\begin{align}\label{eq:PRIdecomposition}
\rmod\Lambda = \kP \vee \kR \vee \kI
\end{align}
into so called \emph{preprojective} resp.~\emph{regular} resp.~\emph{preinjective} $\Lambda$-modules, where non-zero homomorphisms exist only from left to right. The regular modules are precisely those which are mapped to itself under the Auslander-Reiten translation $\tau$. Moreover the subcategory $\kR$ is abelian and consists of $\Hom$- and $\Ext$-orthogonal subcategories 
each of which being a uniserial category with a unique simple object up to isomorphism \cite{DlabRingel}.
We note that $\Lambda$ is the tensor algebra of the $k$-species
\begin{align}\label{eq:species}
\xymatrix
{
F \ar[rr]^{{}_FM_G} && G.
}
\end{align}
A representation of the species \eqref{eq:species} (or simply a representation of the bimodule ${}_FM_G$) is a triple $(U,V,\rho)$ where $U=U_F$ is a right $F$-module, $V=V_G$ is a right $G$-module and $\rho$ is a morphism of right $G$-modules
\[
\rho:~U_F\otimes_F{}_FM_G\longrightarrow V_G.
\]
Given two representations $(U,V,\rho)$ and $(U',V',\rho')$ of $M$, a homomorphism of representations consists of a pair of maps $(f,g)\in\Hom_F(U,U')\times\Hom_G(V,V')$ such that the diagram
\[
\xymatrix{
U\otimes_FM \ar[r]^-\rho \ar[d]_{f\otimes\id_M} & V \ar[d]^g \\
U'\otimes_FM \ar[r]^-{\rho'} & V'
}
\]
is commutative. The category of representations of the species \eqref{eq:species} is equivalent to the category $\rmod\Lambda$ of right $\Lambda$-modules.
More concretely, the $\Lambda$-module associated to a species representation $(U,V,\rho)$ is the $k$-vector space $U\oplus V$ with $\Lambda$-action
\[
(u,v) \cdot
\begin{pmatrix}
f & m \\ 0 & g
\end{pmatrix} 
= (uf, \rho(u\otimes m)+vg)\qquad \begin{array}{l} \forall f\in F, g\in G ,m\in M \\ \forall u\in U, v\in V. \end{array}
\]
We also note that the underlying (oriented) weighted graph of the species \eqref{eq:species} is
\begin{align}\label{eq:Epsilon}
\xymatrix
{
\bullet \ar[rr]^{\left(\frac{2}{\varepsilon},~2\varepsilon\right)} && \bullet
}
\qquad\quad\text{ with }\qquad \varepsilon = \sqrt{\frac{\dim{}_FM}{\dim M_G}}\in\left\{\frac{1}{2}\,,1\,,2\right\}.
\end{align}
It is then clear that the division algebras $F$ and $G$ satisfy
\begin{align}\label{eq:FGdimensions}
\dim_kG=\varepsilon^2\dim_kF.
\end{align}
From the representation theory of tame hereditary algebras we know in which cases an indecomposable species representation $(U,V,\rho)$ is preprojective, regular or preinjective depending on the dimension vector, namely
\begin{align}\label{eq:PRIdecompositionDimVector}
(U,V,\rho)\in 
\begin{cases}
\kP, & \text{ if } \dim U_F<\varepsilon\dim V_G, \\
\kR, & \text{ if } \dim U_F=\varepsilon\dim V_G, \\
\kI, & \text{ if } \dim U_F>\varepsilon\dim V_G.
\end{cases}
\end{align}
In order to define a squid algebra, we need the following input data:
\begin{itemize}
\item a tame bimodule $M={}_FM_G$ over finite-dimensional $k$-division algebras $F$ and $G$ such that $k$ acts centrally on $M$, called \emph{minimal model},
\item integers $t\geq 1$ and $p_1,\dots,p_t\geq 2$, called \emph{weights} and
\item pairwise non-isomorphic regular-simple representations
\[
\rho_i:~U_i\otimes_FM\longrightarrow V_i \qquad \text{ for } 1\leq i\leq t
\]
of the bimodule ${}_FM_G$, called \emph{exceptional points}.
\end{itemize}
A representation $(U_i,V_i,\rho_i)$ of M is called \emph{regular-simple} if it is a simple object in the abelian category of regular representations (since the category of representations of $M$ is equivalent to $\rmod\Lambda$, we can use the decomposition \eqref{eq:PRIdecomposition}). In particular, the endomorphism rings $D_i=\End(\rho_i)$ are finite-dimensional $k$-division algebras. For every $1\leq i\leq t$, we get a left $D_i$-action on $U_i$ and $V_i$ turning them into bimodules such that $\rho_i$ is a morphism of $D_i$-$G$-bimodules.

\begin{definition}
The \emph{squid algebra} associated to the tame bimodule ${}_FM_G$, weights $p_1,\dots,p_t$ and exceptional points $\rho_1,\dots,\rho_t$ is the finite-dimensional $k$-algebra
\begin{align*}
A :=
\left(
\begin{array}{ccc|c|ccc|cc}
D_1 & \cdots & D_1 & 0 & 0 & \cdots & 0 & U_1 & V_1 \\
& \ddots & \vdots & \vdots & \vdots &  & \vdots & \vdots & \vdots \\
&  & D_1 & 0 & 0 & \cdots & 0 & U_1 & V_1 \\
 \hline
 &  &  & \ddots & 0 & \cdots & 0 & \vdots & \vdots \\
 \hline
&  &  &  & D_t & \cdots & D_t & U_t & V_t \\
&  &  &  &  & \ddots & \vdots & \vdots & \vdots \\
&  &  &  &  &  & D_t & U_t & V_t \\
 \hline
 &  &  &  &  &  &  & F & M \\
 &  &  &  &  &  &  &  & G \\
\end{array}
\right),
\end{align*}
where $D_i$ appears precisely $p_i-1$ times on the diagonal. The multiplication in $A$ is given by the bimodule structures over the division algebras respectively by the $D_i$-$G$-bimodule homomorphisms $\rho_i:~U_i\otimes_F M\longrightarrow V_i$.
\end{definition}

\begin{remark}
The squid algebra is a quotient of the tensor algebra of the $k$-species
\begin{align}
\label{eq:speciesSquid}
\begin{minipage}{12cm}
\xymatrix@R=1pc
{
D_1 \ar[r]^{D_1}  & D_1 \ar[r]^{D_1} & \dots & \dots \ar[r]^{D_1} & D_1 \ar[rdd]^{U_1} &  &  \\
D_2 \ar[r]^{D_2}  & D_2 \ar[r]^{D_2} & \dots & \dots \ar[r]^{D_2} & D_2 \ar[rd]_{U_2} &  &  \\
  & \vdots &  & \vdots &  & F \ar[r]^M & G \\
  & \vdots &  & \vdots &  &  &  \\
D_t \ar[r]^{D_t}  & D_t \ar[r]^{D_t} & \dots & \dots \ar[r]^{D_t} & D_t \ar[ruu]_{U_t} &  &  \\
}
\end{minipage}
\end{align}
where the $i$-th arm has $p_i-1$ vertices.
\end{remark}

\begin{example}\label{example:simplyLacedSquid}
If we start with the tame $k$-$k$-bimodule $k^2$, then the exceptional points are regular-simple representations of the Kronecker quiver which are in one-to-one correspondence with homogeneous irreducible polynomials in $k[X,Y]$ up to $k^\times$-action or in other words with closed points of the projective line $\PP_k^1$. Suppose we only choose points $(1\!:\!\lambda_1),\dots,(1\!:\!\lambda_t)$ of degree one (meaning $\lambda_1,\dots,\lambda_t\in k\cup\{\infty\}$). Then, it is easy to see that the corresponding squid algebra $A$ is the path algebra of the quiver
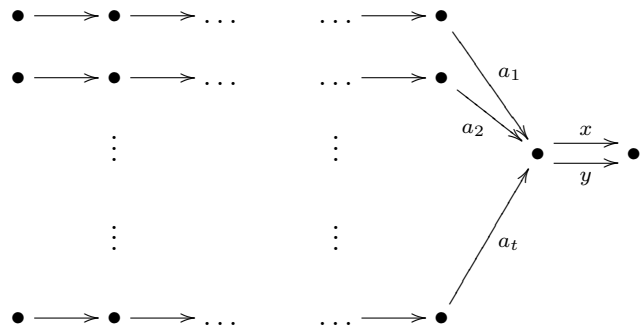
\[
\xymatrix@R=1pc
{
\bullet \ar[r]  & \bullet \ar[r] & \dots & \dots  \ar[r] & \bullet \ar[rdd]^{a_1} & &   \\
\bullet \ar[r]  & \bullet \ar[r] & \dots & \dots  \ar[r] & \bullet \ar[rd]_{a_2} & &   \\
  & \vdots &  & \vdots &    & \bullet \ar@<0.8ex>[r]^x \ar@<-0.8ex>[r]_y & \bullet  \\
  & \vdots &  & \vdots &   & &   \\
\bullet \ar[r]  & \bullet \ar[r] & \dots & \dots  \ar[r] & \bullet \ar[ruu]_{a_t} & &   \\
}
\]
where the $i$-th arm has $p_i-1$ vertices, modulo the relations
\begin{align*}
a_i(x+\lambda_iy)=0\qquad\forall~ 1\leq i\leq t.
\end{align*}
Note that for an algebraically closed field $k$, the only finite-dimensional division algebra is the field $k$ itself, the only tame $k$-$k$-bimodule with central $k$-action is $k^2$ and every irreducible polynomial over $k$ has degree one. Therefore this example covers all possible cases over an algebraically closed field.
\end{example}
Before we continue, we need to fix some more notation. A complete set of primitive orthogonal idempotents in $A$ is given by those matrices with a unique entry $1$ on the diagonal and only zeroes elsewhere. We write
\[
e_1(p_1-1),\dots,e_1(1),\dots\quad\dots,e_t(p_t-1),\dots,e_t(1),e_F,e_G
\]
for the idempotent elements in $A$, meaning $e_i(j)$ is the matrix with entry $1_{D_i}$ at the $(p_i-j)$-th position in the $i$-th block and zeroes elsewhere. This notation corresponds to a labelling of the vertices in the $i$-th arm with number $p_i-1$ to $1$ going from left to right.

Our goal is to define another finite-dimensional algebra which -- as we will show in the next section -- is derived equivalent to the squid algebra $A$, the so called \emph{Coxeter-Dynkin algebra of canonical type}. In the simply-laced case, a definition as a path algebra of a quiver modulo relations can be found in \cite{LenzingdelaPenaExtendedCanonical} as well as in \cite{WhatYouDidntWantToKnow}. Moreover, both references contain a concrete description of a tilting object in the category $\Coh(\XX)$ of coherent sheaves on a weighted projective line whose endomorphism algebra is a Coxeter-Dynkin algebra. However, both the definition of the Coxeter-Dynkin algebra and the construction of the tilting object rely on the assumption that there are at least two exceptional points. In the general situation this condition has to be formulated differently. For this purpose we introduce the notation
\begin{align*}
U_i^\vee := \Hom_{D_i}\left({}_{D_i}{U_i}_F,{}_{D_i}{D_i}_{D_i}\right)
\end{align*}
for the $D_i$-dual of $U_i$, which is an $F$-$D_i$-bimodule. Moreover, we note that there are isomorphisms of $F$-$G$-bimodules
\begin{align*}
U_i^\vee\otimes_{D_i}V_i&\stackrel{\cong}{\longrightarrow}\Hom_{D_i}(U_i,V_i)\\
\varphi\otimes v &\longmapsto \varphi(-)v
\end{align*}
for all $1\leq i\leq t$.
Furthermore, we define the finite-dimensional $k$-algebra
\begin{align*}
A_0=e_0Ae_0\qquad\quad\text{with}\qquad\quad e_0=1-e_G \in A
\end{align*}
which is given by the tensor algebra of the species \eqref{eq:speciesSquid} with the rightmost vertex deleted. In particular, $A_0$ is hereditary and $A$ can be written as a one-point-extension
\begin{align*}
A=\begin{pmatrix}
A_0 & N \\ 0 & G
\end{pmatrix}
\end{align*}
with respect to the left $A_0$-module $N=e_0Ae_G\in A_0 \lmod$. We are now ready to formulate conditions under which a reasonable notion of Coxeter-Dynkin algebra can be defined. Later, we will encounter two further equivalent conditions in the context of canonical algebras (see Remark \ref{rmk:CanonicalAlgebraSpecies}) and the representation type (see Remark \ref{rmk:RepFiniteTiltedAlgebras}) which do not play a role at the moment.

\begin{proposition}\label{Prop:OctopusCondition}
The following conditions are equivalent:
\begin{enumerate}
\item[(1)] The morphism of $F$-$\,G$-bimodules
\begin{align*}
\theta_0:\quad{}_FM_G&\longrightarrow \bigoplus_{i=1}^t \Hom_{D_i}\left(U_i,V_i\right)\cong \bigoplus_{i=1}^t U_i^\vee\otimes_{D_i}V_i \\
m & \longmapsto \left(\rho_i(-\otimes m)\right)_{1\leq i\leq t}
\end{align*}
is a monomorphism.
\item[(2)] $\pdim(\tau^{-1}(\tp(e_FA)))=1$.
\item[(3)] $\Hom_{\rmod A}(\sD({}_AA), \tp(e_FA))=0$.
\item[(4)] $\Hom_{A_0\lmod}(\tp(A_0e_F), N)=0$.
\item[(5)] The left $A_0$-module $N$ has no direct summand isomorphic to $\tp(A_0e_F)$.
\item[(6)] The following inequality is satisfied:
\begin{align*}
\sum_{i=1}^t\dim{}_{D_i}(U_i)\cdot\dim(U_i)_F=
\sum_{i=1}^t\dim{}_{D_i}(V_i)\cdot\dim(V_i)_G\geq 2.
\end{align*}
\end{enumerate}
\end{proposition}
\begin{proof} \underline{(1)$\Leftrightarrow$(2):}
We use that $\tau^{-1}(\tp(e_FA))$ can be computed as the transpose of the left $A$-module $\sD(\tp(e_FA)) \cong \tp(Ae_F)$. A minimal projective presentation of $\tp(Ae_F)$ is given by
\begin{align}\label{eq:projPresTop(Ae_F)}
\bop_{i=1}^t Ae_i(1)\otimes_{D_i} U_i \stackrel{\pi}{\longrightarrow} Ae_F \longrightarrow \tp(Ae_F) \longrightarrow 0,
\end{align}
where $\pi$ corresponds to $\left(\id_{U_i}\right)_{1\leq i\leq t}$ under the isomorphism
\begin{align*}
\Hom_A\left( \bop_{i=1}^t Ae_i(1)\otimes_{D_i} U_i , Ae_F \right)
\cong \bop_{i=1}^t \Hom_{D_i}( U_i , \underbrace{\Hom_A(Ae_i(1),Ae_F)}_{\cong U_i}).
\end{align*}
Applying the functor $\Hom_A(-,A):~A\lmod \longrightarrow \rmod A$ to \eqref{eq:projPresTop(Ae_F)}, yields the exact sequence in the first row of the diagram
\begin{align*}
\xymatrix@R=1pc{
  e_FA \ar[r] \ar@{=}[ddd] & \Hom_A\left( \bop_{i=1}^t Ae_i(1)\otimes_{D_i} U_i , A\right) \ar[r] \ar[d]^\cong & \mathsf{Tr}(\tp(Ae_F)) \ar[r] \ar@{=}[ddd] & 0 \\
   & \bop_{i=1}^t \Hom_{D_i}\left(U_i , \Hom_A(Ae_i(1),A)\right) \ar[d] ^\cong & & \\
     & \bop_{i=1}^t \Hom_{D_i}\left(U_i , e_i(1)A\right) & & \\
  e_FA \ar[r]^-\theta & \bop_{i=1}^t U_i^\vee\otimes_{D_i}e_i(1)A \ar[u]_\cong \ar[r] & \tau^{-1}(\tp(e_FA)) \ar[r] & 0 
}
\end{align*}
where the lowest vertical map in the middle column is the isomorphism $\phi\otimes a\mapsto \phi(-)a$. Furthermore, $\theta$ is defined as the morphism which corresponds to the tuple $\left(\id_{U_i}\right)_{1\leq i\leq t}$ under the isomorphism
\begin{align*}
\Hom_A\left( e_FA,~\bigoplus_{i=1}^t U_i^\vee\otimes_{D_i} e_i(1)A \right)
\cong \bigoplus_{i=1}^t U_i^\vee\otimes_{D_i}\underbrace{e_i(1)Ae_F}_{\cong U_i}
\cong \bigoplus_{i=1}^t \Hom_{D_i} \left(U_i,U_i\right).
\end{align*}
A direct computation then shows that the diagram is commutative. Note that the above right exact sequence is a minimal projective presentation of $\tau^{-1}(\tp(e_FA))$. Hence, (2) is equivalent to $\theta$ being a monomorphism. Since the only proper submodules of $e_FA$ isomorphic to direct sums of the simple module $e_GA$, it suffices to restrict the map $\theta$ to the rightmost vertex in the species \eqref{eq:speciesSquid}. More precisely, $\theta$ is a monomorphism if and only if $\Hom_A(e_GA,\theta)$ is a monomorphism. Now the claim follows since $\Hom_A(e_GA,\theta)$ coincides with $\theta_0$ up to the canonical isomorphisms.
\\ \underline{(2)$\Leftrightarrow$(3):} This follows from \cite[Lemma IV.2.7]{AssemSimsonSkowronski}.
\\ \underline{(3)$\Leftrightarrow$(4):} Condition (3) is equivalent to
\begin{align*}
\Hom_{A \lmod}(\tp(Ae_F),P)=0
\end{align*}
for every indecomposable projective left $A$-module $P$.
The $A$-$A_0$-bimodule $Ae_0$ gives rise to an adjoint pair of functors
\begin{align*}
\xymatrix@C=3pc
{
A_0\lmod \ar@/^1pc/[rr]^{Ae_0\otimes_{A_0}-} & \perp & A\lmod. \ar@/^1pc/[ll]^{\Hom_{A\lmod}(Ae_0,-)}
}
\end{align*}
In particular, we get
\begin{align*}
\Hom_{A\lmod}(\tp(Ae_F),Ae_G) &\cong \Hom_{A\lmod}(Ae_0\otimes_{A_0}\tp(A_0e_F),Ae_G)\\
&\cong \Hom_{A_0\lmod}(\tp(A_0e_F),\Hom_A(Ae_0, Ae_G))\\
&\cong \Hom_{A_0\lmod}(\tp(A_0e_F),e_0Ae_G)\\
&= \Hom_{A_0\lmod}(\tp(A_0e_F),N),
\end{align*}
such that (3) is clearly sufficient for (4). The implication "(4)$\Rightarrow$(3)" holds because in general, $\Hom_{A\lmod}(\tp(Ae_F),P)$ vanishes for all indecomposable projective left $A$-modules except $Ae_G$.
\\ \underline{(4)$\Leftrightarrow$(5):} Clear, because $\tp(A_0e_F)$ is a simple, injective left $A_0$-module.
\\ In regard of condition (6) we note that for every regular-simple representation $(U,V,\rho)$ with endomorphism algebra $D$, the following holds:
\begin{align}
\nonumber
\dim{}_DU&=\frac{\dim_kU}{\dim_kD}
=\frac{\dim_kF\cdot\dim U_F}{\dim_kD}
\stackrel{\eqref{eq:FGdimensions},\eqref{eq:PRIdecompositionDimVector}}{=}\frac{\dim_kG\cdot\dim V_G}{\varepsilon\dim_kD}\\ \nonumber
&=\frac{\dim_kV}{\varepsilon\dim_kD}
=\frac{1}{\varepsilon}\dim{}_DV\\[+0.5cm] \nonumber
\Rightarrow
\dim{}_F(U^\vee\otimes_D V)
&= \dim{}_F(U^\vee)\cdot\dim{}_DV
= \dim U_F\cdot\dim{}_DV \\ \label{eq:dimUTensorV}
&= \varepsilon \dim {}_DU\cdot\dim U_F
= \varepsilon \dim {}_DV\cdot\dim V_G.
\end{align}
In particular, the equality in (6) holds in general.
\\ \underline{(1)$\Rightarrow$(6):} The $F$-dimension of $M$ is given by
\begin{align*}
\dim{}_FM = \sqrt{\dim{}_FM\cdot\dim M_G}\sqrt{\frac{\dim{}_FM}{\dim M_G}} = 2\varepsilon.
\end{align*}
If we assume that $\theta_0$ is injective then necessarily
\begin{align*}
2\varepsilon &= \dim {}_FM \leq \dim{}_F\left(\bop_{i=1}^t U_i^\vee\otimes_{D_i}V_i\right)\\
&= \sum_{i=1}^t \dim{}_F\left(U_i^\vee\otimes_{D_i}V_i\right)
\stackrel{\eqref{eq:dimUTensorV}}{=} \varepsilon \sum_{i=1}^t \dim {}_D(U_i)\cdot\dim (U_i)_F,
\end{align*}
which proves (6).
\\ \underline{(6)$\Rightarrow$(1):} We start with the special case that there is only one exceptional point $\rho=\rho_1$ and assume that the map
\begin{align*}
\theta_0=\rho^\#:\quad M\longrightarrow \Hom_D(U,V),\quad m \longmapsto \rho(-\otimes m)
\end{align*}
is not a monomorphism. Then its kernel must be a proper $F$-$G$-subbimodule $0\neq K\subsetneq M$. Such a submodule can only exist in the case $\varepsilon=1$ since otherwise $M$ is $1$-dimensional over $F$ or over $G$. Hence, we can deduce $\varepsilon=1$ and $\dim {}_FM=2=\dim M_G$. Following the strategy in \cite[p.~108]{DlabLectureNotes}, we can identify $K$ with ${}_FF_G$ carrying an $F$-$G$-bimodule structure and observe that
\begin{align*}
G\longrightarrow F, \quad g\longmapsto 1_F\cdot g
\end{align*}
is an isomorphism of $k$-algebras. Under this isomorphism the bimodule ${}_FF_G$ identifies with the $F$-$F$-bimodule ${}_FF_F$ with standard $F$-actions. Therefore, we can assume that $F=G$ and
\begin{align*}
{}_FM = Fs\oplus Ft
\end{align*}
as left $F$-modules (with $K=Fs$). The right $F$-action can be written as
\begin{align*}
sf = fs \qquad\text{and} \qquad tf = \delta(f)s+\varphi(f)t\qquad \forall f\in F
\end{align*}
for some $k$-linear maps $\delta,\varphi:~F\longrightarrow F$. For all $f_1,f_2\in F$, these maps must satisfy
\begin{align*}
\delta(f_1f_2)s+\varphi(f_1f_2)t
&= t(f_1f_2) = (tf_1)f_2
= \delta(f_1)sf_2+\varphi(f_1)tf_2\\
&= \delta(f_1)f_2s+\varphi(f_1)(\delta(f_2)s+\varphi(f_2)t)\\
&= \left(\delta(f_1)f_2+\varphi(f_1)\delta(f_2)\right)s+\varphi(f_1)\varphi(f_2)t\\[+2pt]
\Rightarrow \varphi(f_1f_2) &= \varphi(f_1)\varphi(f_2) \qquad\text{and}\\
\delta(f_1f_2) &= \delta(f_1)f_2+\varphi(f_1)\delta(f_2).
\end{align*}
It is also clear that $\varphi(1)=1$. Thus, $\varphi$ is a $k$-algebra homomorphism which must be bijective because $F$ is a finite-dimensional division algebra. Moreover $\delta$ is a so-called $(1,\varphi)$-derivation of $F$. Let $0\neq u\in U$ and $v=\rho(u\otimes t)\in V$, then we claim that $(uF,vF,\rho|_{uF})$ is a subrepresentation of $(U,V,\rho)$. Indeed, given $m\in M$ we can write it as $m=f_1s+f_2t$ for some $f_1,f_2\in F$ and get
\begin{align*}
\rho(u\otimes m)&=\rho(u\otimes f_1s)+\rho(u\otimes f_2t)
=\rho(u\otimes sf_1)+\rho\left(u\otimes (t\varphi^{-1}(f_2)-\delta\varphi^{-1}(f_2)s)\right) \\
&=\underbrace{\rho(u\otimes s)}_{=0}f_1
+\underbrace{\rho(u\otimes t)}_{=v}\varphi^{-1}(f_2)-\underbrace{\rho(u\otimes s)}_{=0}\delta\varphi^{-1}(f_2)
\quad\in vF.
\end{align*}
Hence $(uF,vF,\rho|_{uF})$ is a subrepresentation which is regular by \eqref{eq:PRIdecompositionDimVector}. Since $(U,V,\rho)$ is assumed to be regular-simple, the above subrepresentation cannot be proper and we can conclude $U=uF$ and $V=vF$.  We get
\begin{align*}
\dim U_F = \dim {}_DU = 1 = \dim V_G = \dim {}_DV
\end{align*}
and have shown in the case $t=1$ that (6) cannot be true if (1) fails.
\\ We proceed with the general case and assume again that $\theta_0$ is not injective. It follows that neither of the maps
\begin{align*}
\rho_i^\#:\quad M\longrightarrow \Hom_D(U_i,V_i),\quad m \longmapsto \rho_i(-\otimes m)
\end{align*}
is injective. By the consideration in the special case, we know that we can assume $F=G$ as well as $U_i=F=V_i$ for all $1\leq i\leq t$. Moreover, since
\begin{align*}
0\neq \ker(\theta_0)=\bigcap_{i=1}^t\ker\left(\rho_i^\#\right)
\end{align*}
and every kernel is a $1$-dimensional $F$-$F$-subbimodule of $M$, we get $\ker(\rho_i^\#)=\ker(\rho_j^\#)$ for all $1\leq i,j\leq t$. The commutative diagram of $F$-$F$-bimodules
\begin{align*}
\xymatrix{
\ker(\rho_i) \ar@{^{(}->}[r] \ar[d] & F\otimes M \ar@{->>}[r]^{\rho_i} \ar[d]_{\can} & F \ar[d] & \ni f\quad \ar@{|->}[d] \\
\ker\left(\rho_i^\#\right) \ar@{^{(}->}[r] & M \ar@{->>}[r]^-{\rho_i^\#} & \Hom({}_FF,{}_FF) & \ni (a\mapsto af)
}
\end{align*}
shows that $\ker(\rho_i)=F\otimes\ker(\rho_i^\#)$ and therefore $\ker(\rho_i)=\ker(\rho_j)$ for all $1\leq i,j\leq t$. Since the representations $(U_i,V_i,\rho_i)$ are uniquely determined -- up to isomorphism -- by $\ker(\rho_i)$, they are all isomorphic. Hence, we can conclude $t=1$ because we assumed the exceptional points to be pairwise non-isomorphic representations. This finishes the proof that (6) fails if $\theta_0$ is not a monomorphism.
\end{proof}

In the following, we define a Coxeter-Dynkin algebra starting with a tame $F$-$G$-bimodule ${}_FM_G$, and a choice of weights $p_1,\dots,p_t\geq 2$ and exceptional points $\rho_1,\dots,\rho_t$ such that the conditions from Proposition \ref{Prop:OctopusCondition} are satisfied. In other words, we only exclude the case $F=G$, $t=1$ and $U_1=F_F=V_1$.

\begin{definition}\label{Def:CoxeterDynkinAlgebra}
Let ${}_FM_G$ be a tame bimodule, $p_1,\dots,p_t\geq 2$ weights and $\rho_1,\dots,\rho_t$ exceptional points such that the conditions from Proposition \ref{Prop:OctopusCondition} are satisfied.
Let
\begin{align*}
{}_FW_G = \cok\left(
\xymatrix{
{}_FM_G \ar^-{\theta_0}[r] & \bigoplus_{i=1}^t \Hom_{D_i}\left(U_i,V_i\right)\cong \bigoplus_{i=1}^t U_i^\vee\otimes_{D_i}V_i
}
\right),
\end{align*}
then the \emph{Coxeter-Dynkin algebra of canonical type} associated to this data is the finite-dimensional $k$-algebra
\begin{align*}
B :=
\left(
\begin{array}{c|cccc|c|cccc|c}
F & 0 & \cdots & 0 & U_1^\vee & \cdots & 0 & \cdots & 0 & U_t^\vee & W \\
\hline
 & D_1 & \cdots & D_1 & D_1 & 0 & 0 & \cdots & 0 & 0 & V_1 \\
 &  & \ddots & \vdots & \vdots & \vdots & \vdots &  & \vdots & \vdots  & \vdots \\
 &  &  & D_1 & D_1 & 0 & 0 &\cdots & 0 & 0 & V_1 \\
 &  &  &  & D_1 & 0 &  0 & \cdots & 0 & 0 & V_1 \\
\hline
 &  &  &  &  & \ddots & 0 & \cdots & 0 & 0 & \vdots \\
\hline
 &  &  &  &  &  & D_t & \cdots & D_t & D_t & V_t \\
  &  &  &  &  &  &  & \ddots & \vdots & \vdots & \vdots \\
 &  &  &  &  &  &  &  & D_t & D_t & V_t \\
 &  &  &  &  &  &  &  &  & D_t & V_t \\
\hline
 &  &  &  &  &  &  &  &  &  & G 
\end{array}
\right),
\end{align*}
where $D_i$ appears precisely $p_i-1$ times on the diagonal. The multiplication in $B$ is given by the bimodule structures over the division algebras respectively by the $F$-$G$-bimodule homomorphisms
\begin{align*}
\xymatrix{
U_i^\vee\otimes_{D_i} V_i \ar@{^{(}->}[r] & \bigoplus_{j=1}^tU_j^\vee\otimes_{D_j} V_j \ar@{->>}[r] & W.
}
\end{align*}
\end{definition}

\begin{remark}
The Coxeter-Dynkin algebra of canonical type is a quotient of the tensor algebra of the $k$-species
\begin{align}\label{eq:speciesOctopus}
\begin{minipage}{12cm}
\xymatrix@R=1pc
{
D_1 \ar[r]^{D_1}  & D_1 \ar[r]^{D_1} & \dots & \dots \ar[r]^{D_1} & D_1 \ar[r]^{D_1} & D_1 \ar[rdd]^{V_1} &   \\
D_2 \ar[r]^{D_2}  & D_2 \ar[r]^{D_2} & \dots & \dots \ar[r]^{D_2} & D_2 \ar[r] & D_2 \ar[rd]_{V_2} &   \\
  & \vdots &  & \vdots & F \ar[ruu]^(.7){U_1^\vee} \ar[ru]_{U_2^\vee} \ar[rdd]_{U_t^\vee} &  \vdots  & G  \\
  & \vdots &  & \vdots &  & \vdots &   \\
D_t \ar[r]^{D_t}  & D_t \ar[r]^{D_t} & \dots & \dots \ar[r]^{D_t} & D_t \ar[r]^{D_t} & D_t \ar[ruu]_{V_t} &   \\
}
\end{minipage}
\end{align}
where the $i$-th arm has $p_i-1$ vertices.
\end{remark}

\begin{example}\label{example:simplyLacedCDalgebra}
If we start with the same initial data as in Example \ref{example:simplyLacedSquid}, then the corresponding Coxeter-Dynkin algebra $B$ is the path algebra of the quiver
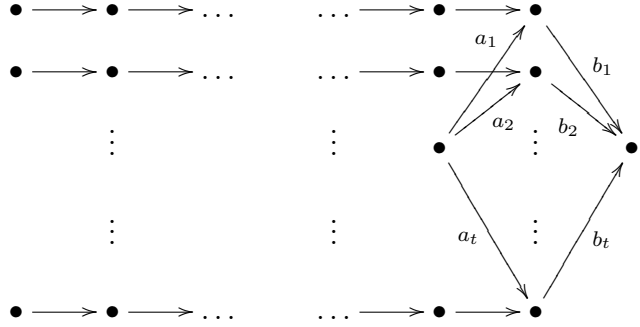
\[
\xymatrix@R=1pc
{
\bullet \ar[r]  & \bullet \ar[r] & \dots & \dots \ar[r] & \bullet \ar[r] & \bullet \ar[rdd]^{b_1} &   \\
\bullet \ar[r]  & \bullet \ar[r] & \dots & \dots \ar[r] & \bullet \ar[r] & \bullet \ar[rd]_{b_2} &   \\
  & \vdots &  & \vdots & \bullet \ar[ruu]^(.7){a_1} \ar[ru]_{a_2} \ar[rdd]_{a_t} &  \vdots  & \bullet  \\
  & \vdots &  & \vdots &  & \vdots &   \\
\bullet \ar[r]  & \bullet \ar[r] & \dots & \dots \ar[r] & \bullet \ar[r] & \bullet \ar[ruu]_{b_t} &   \\
}
\]
where the $i$-th arm has $p_i-1$ vertices, modulo the relations
\begin{align*}
\sum_{i=1}^ta_ib_i = 0\qquad\text{ and }\qquad \sum_{i=1}^t\lambda_i a_ib_i = 0.
\end{align*}
Hence, it is isomorphic to the Coxeter-Dynkin algebra of canonical type defined in
\cite{LenzingdelaPenaExtendedCanonical} and our definition is indeed a generalisation of the existing one. Again this example covers all possible cases if the base field $k$ is algebraically closed.
\end{example}

\begin{remark}
Coxeter-Dynkin algebras (over an algebraically closed base field) also play a role in \cite{ShiraishiTakahashiWadaTerminologyOctopus}. There, the authors instead use the term \emph{octopus algebra}.
\end{remark}

\begin{example}\label{example:complexNumbers}
Let $k=\RR$ and $M=\RR^2$. Let $\rho:\CC\otimes_\RR M \longrightarrow \CC$ be a regular-simple representation of the Kronecker quiver over $\RR$ with $\End(\rho)\cong\CC$, i.\,e., we can assume that $\rho=\rho_{a+ib}$ with
\begin{align*}
\rho_{a+ib}\left(1\otimes \begin{pmatrix}1 \\ 0\end{pmatrix} \right) = 1,\qquad
\rho_{a+ib}\left(1\otimes \begin{pmatrix}0 \\ 1\end{pmatrix} \right) = a+ib,
\end{align*}
for some $a\in\RR$ and $b\in\RR_{>0}$ (for $b=0$, this representation would not be regular-simple, moreover the representations $\rho_{a+ib}$ and $\rho_{a-ib}$ are isomorphic). For simplicity, let us assume that the corresponding weight is $p=p_1=2$.
Now, the squid algebra $A$ and the Coxeter-Dynkin algebra $B$ of canonical type can be written as 
\begin{align*}
A=\begin{pmatrix}
\CC & \CC & \CC \\
0 & \RR & \,\RR^2 \\
0 & 0 & \RR
\end{pmatrix},
\qquad\qquad
B=\begin{pmatrix}
\RR & \CC & 0 \\
0 & \CC & \CC \\
0 & 0 & \RR
\end{pmatrix}.
\end{align*}
In particular, the Coxeter-Dynkin algebra is a quotient of
\begin{align*}
C=\begin{pmatrix}
\RR & \CC & \CC \\
0 & \CC & \CC \\
0 & 0 & \RR
\end{pmatrix}
\subseteq \mathsf{Mat}_{3\times 3}(\CC),
\end{align*}
which is a tame hereditary algebra of type $\widetilde{B}_2$. One can show, that in this particular case the corresponding squid algebra and Coxeter-Dynkin algebra do not depend on the complex number $a+ib$ up to isomorphism.
We also note that $B$ is of finite representation type, there exist precisely seven isomorphism classes of indecomposable $B$-modules. 
\end{example}

\section{From the squid to the Coxeter-Dynkin algebra via generalised APR-tilting}\label{section:squidToCDviageneralisedAPR}

The objective of this section is to establish a derived equivalence between the squid algebra $A$ and the Coxeter-Dynkin algebra of canonical type $B$ by finding a (classical) tilting module over $A$ whose endomorphism algebra is isomorphic to $B$. We start with recalling some basics of tilting theory. Let $A$ be any finite-dimensional $k$-algebra of finite global dimension. We call a module $T\in \rmod A$ a \emph{classical tilting module} if the following conditions are satisfied (see e.\,g.~\cite{AssemSimsonSkowronski}):
\begin{enumerate}
\item[(CT1)] $T$ has projective dimension at most 1,
\item[(CT2)] $\Ext_A^1(T,T)=0$,
\item[(CT3)] there exists an exact sequence
\[
0\longrightarrow A_A\longrightarrow T_0 \longrightarrow T_1 \longrightarrow 0
\]
with $T_0,T_1\in \add(T)$.
\end{enumerate}
Note that one can find deviating definitions in the literature. Often one requires only that $T$ has finite projective dimension and that there is an $\add(T)$-coresolution of $A_A$ of arbitrary finite length \cite{BrennerButlerHandbookOfTiltingTheory}.
Let us also recall a consequence of the definition. Defining $B=\End_A(T)$, it is clear that $T$ becomes a $B$-$A$-bimodule and we obtain a functor
\begin{align*}
\Hom_A(T,-):\qquad \rmod A \longrightarrow \rmod B,
\end{align*}
which restricts to an equivalence $\add(T)\longrightarrow \mathsf{proj-}B$ of additive categories. In particular, we get an equivalence between the bounded homotopy categories and a diagram
\begin{align*}
\xymatrix{
\kK^b(\add(T)) \ar[r]^\cong \ar[d]_\cong & \kK^b(\mathsf{proj-}B) \ar[d]^\cong \\
\kD^b(\rmod A) \ar@{.>}[r] & \kD^b(\rmod B),
}
\end{align*}
where the right vertical functor is an equivalence because $B$ has finite global dimension. Moreover, the left vertical functor is fully-faithful by condition (CT1) and (CT2) and
its essential image contains complexes of projectives by (CT3).
Hence, it is also an equivalence since $A$ is of finite global dimension
\cite[§III.2]{HappelTriangulated}. In other words, in this situation we obtain an equivalence (of triangulated categories) between $\kD^b(\rmod A)$ and $\kD^b(\rmod B)$ and we say that $A$ and $B$ are \emph{derived equivalent}.

One approach to construct tilting modules is to start with the direct sum of the indecomposable projectives and replace one summand by a different indecomposable module. One instance of this is the following result due to Brenner and Butler which generalises the concept of \emph{APR-tilting modules} \cite{APRTilting} (which is recovered in the case where $eA$ is a simple projective $A$-module).
\begin{theorem}[{\cite[Theorem IX]{BrennerButlerReflectionFunctors}}] \label{thm:BBgeneralisedAPR} 
Let $A$ be a basic algebra, $e\in A$ a primitive idempotent and $S=\tp(eA)$ the corresponding simple module such that
\begin{enumerate}
\item[(i)] $S$ is not injective,
\item[(ii)] $\Hom_A(\sD({}_AA),S)=0$ and
\item[(iii)] $\sD(Ae)$ is not a direct summand of the injective envelope of $\sD(Ae)/S$.
\end{enumerate}
Then
\begin{align*}
T= (1-e)A \oplus \tau^{-1}(S) \quad\in \rmod A
\end{align*}
is a classical tilting module.
\end{theorem}

From now on, let $A$ and $B$ be again the squid resp.~Coxeter-Dynkin algebra defined in Section \ref{section:squidAndCDalgebra}. In particular we assume that the conditions from Proposition \ref{Prop:OctopusCondition} are satisfied. We write $X=\tau^{-1}(\tp(e_FA))\in \rmod A$.

\begin{corollary}\label{cor:octopusGenAPRTiltingModule}
The module
\[
T=e_GA~\oplus~ \bigoplus_{i=1}^t\bigoplus_{j=1}^{p_i-1}~ e_i(j)A~ \oplus X \quad\in \rmod A
\]
is a classical tilting module.
\end{corollary}
\begin{proof}
We will apply Theorem \ref{thm:BBgeneralisedAPR}. It is clear that $\tp(e_FA)$ is not injective, since the corresponding vertex in the species \eqref{eq:speciesSquid} is not a source. Condition (ii) is satisfied by assumption as it appears in Proposition \ref{Prop:OctopusCondition}. Finally, the third condition from Theorem \ref{thm:BBgeneralisedAPR} holds for any triangulable algebra (see \cite[§3.3]{BrennerButlerReflectionFunctors}).
\end{proof}

In order to compute the endomorphism algebra of the tilting module $T$, we recall from the proof of Proposition \ref{Prop:OctopusCondition} that $X$ fits into a short exact sequence
\begin{align}
\label{eq:sesX}
0\longrightarrow e_FA \stackrel{\theta}{\longrightarrow} \bop_{i=1}^t U_i^\vee\otimes_{D_i}e_i(1)A \longrightarrow X \longrightarrow 0
\end{align}
and start with the following auxiliary lemma.

\begin{lemma}\label{lemma:PropertiesX}
The module $X=\tau^{-1}(\tp(e_FA))\in \rmod A$ satisfies
\begin{align*}
\tag{i}
\Hom_A(X,e_GA)&=0,\\
\tag{ii}
\Hom_A(X,e_i(j)A)&=0\qquad \forall\; 1\leq i\leq t,\; 1\leq j\leq p_i-1\qquad\text{and} \\
\tag{iii}
\End_A(X)&\cong F.
\end{align*}
\end{lemma}
\begin{proof}
Applying $\Hom_A(-,e_GA)$ to \eqref{eq:sesX} yields an exact sequence
\begin{align*}
0\longrightarrow \Hom_A(X,e_GA)&\longrightarrow \Hom_A\left(\bop_{i=1}^t U_i^\vee \otimes e_i(j)A,e_GA\right)=0
\end{align*}
which proves (i).\\
If we apply $\Hom_A(-,e_i(j)A)$ to \eqref{eq:sesX}, we get a long exact sequence including the top map in the following diagram:
\begin{align*}
\xymatrix{
\Hom_A\left(\bop_{l=1}^t U_l^\vee \otimes_{D_l}e_l(1)A,e_i(j)A\right) \ar[r] \ar[d]_\cong & \Hom_A(e_FA,e_i(j)A) \ar[dd]^\cong \\
\Hom_{D_i}\left(U_i^\vee,\Hom_A(e_i(1)A,e_i(j)A)\right) \ar[d]_\cong & \\
\Hom_{D_i}\left(U_i^\vee,D_i\right) \ar[r]^\cong & U_i
}
\end{align*}
The vertical maps are the obvious isomorphisms and we use that $\Hom_A(e_l(1)A,e_i(j)A)=0$ whenever $l\neq i$. The map at the bottom is the canonical isomorphism $\left(U_i^\vee\right)^\vee\cong U_i$ and one can check by hand that the diagram is commutative. Therefore the top map, whose kernel is $\Hom_A(X,e_i(j)A)$ by the long exact sequence, is an isomorphism and we can deduce (ii).\\
The endomorphism algebra of $X$ in $\rmod A$ can be identified with the endomorphism algebra of the corresponding stalk complex in $\kD^b(\rmod A)$. By \eqref{eq:sesX}, this is isomorphic to the endomorphism algebra of the complex
\begin{align}\label{eq:2termComplex}
e_FA \longrightarrow \bigoplus_{i=1}^t U_i^\vee\otimes_{D_i}e_i(1)A
\end{align}
in $\kK^b(\mathsf{proj-}A)$. Since $\Hom_A(e_i(1)A,e_FA)=0$ for all $1\leq i\leq t$, there are no nullhomotopies of non-trivial chain maps and we only have to compute the endomorphisms of \eqref{eq:2termComplex} in $\kC^b(\mathsf{proj-}A)$. For this, we first claim that such an endomorphism
\begin{align}\label{eq:2termComplexEnd}
\begin{minipage}{12cm}
\xymatrix
{
e_FA \ar[r] \ar[d]_f & \bigoplus_{i=1}^t U_i^\vee\otimes_{D_i}e_i(1)A \ar[d]^g \\
e_FA \ar[r] & \bigoplus_{i=1}^t U_i^\vee\otimes_{D_i}e_i(1)A 
}
\end{minipage}
\end{align}
is uniquely determined by $f\in \End_A(e_FA)\cong F$. Indeed if $f$ is zero in \eqref{eq:2termComplexEnd} then so is $g$ since $\Hom_A(X,e_i(1)A)=0$ by (ii). Secondly, for every $f\in\End_A(e_FA)\cong F$ there exists a map $g$ such that \eqref{eq:2termComplexEnd} commutes, because $U_i^\vee$ carries a left $F$-action for every $1\leq i\leq t$. This proves that the 2-term complex \eqref{eq:2termComplex} has the endomorphism ring $F$ in $\kC^b(\mathsf{proj-}A)$ and thus $\End_A(X)\cong F$ as $k$-algebras.
\end{proof}

Note that one could have also used properties of the Auslander-Reiten translation to compute the endomorphism algebra of  $X=\tau^{-1}(\tp(e_FA))$. Since $\tau$ and $\tau^{-1}$ induce mutually inverse equivalences between stable categories \cite[IV.2.11]{AssemSimsonSkowronski}, one only has to argue that the stable endomorphism algebra coincides with the actual endomorphism algebra.

\begin{theorem}\label{thm:tiltSquidToOctopus}
The $A$-module
\[
T=e_GA~\oplus~ \bigoplus_{i=1}^t\bigoplus_{j=1}^{p_i-1}~ e_i(j)A~ \oplus X \quad\in \rmod A
\]
is a classical tilting module with $\End_A(T) \cong B$. In particular, the functor $\Hom_A(T,-)$ induces an equivalence of triangulated categories
\begin{align*}
\kD^b(\rmod A)\stackrel{\cong}{\longrightarrow} \kD^b(\rmod B).
\end{align*}
\end{theorem}
\begin{proof}
The fact that $T$ is a classical tilting module has already been deduced in Corollary \ref{cor:octopusGenAPRTiltingModule} from Theorem \ref{thm:BBgeneralisedAPR}. It remains to compute the endomorphism algebra of $T$.\\
We can apply $\Hom_A(e_i(j)A,-)$ to \eqref{eq:sesX} and use that $\Hom_A(e_i(j)A,e_FA)=0$, to obtain
\begin{align*}
\Hom_A(e_i(j)A,X)&\cong \Hom_A\left(e_i(j)A~,~\bop_{l=1}^t U_l^\vee\otimes_{D_l}e_l(1)A\right)\\[+0.3cm]
&= \Hom_A\left(e_i(j)A,U_i^\vee\otimes_{D_i}e_i(1)A\right)
\cong \begin{cases} U_i^\vee\otimes_{D_i} D_i\cong U_i^\vee, & \text{ if } j=1, \\ 0, & \text{ if } j\neq 1. \end{cases}
\end{align*}
If we apply $\Hom_A(e_GA,-)$ to \eqref{eq:sesX} we get the short exact sequence
\begin{align*}
0\longrightarrow M \stackrel{\theta_0}{\longrightarrow} \bigoplus_{i=1}^t U_i^\vee\otimes_{D_i}V_i \longrightarrow \Hom_A(e_GA,X) \longrightarrow 0,
\end{align*}
which implies that $\Hom_A(Ae_G,X)\cong W$. With the previous two calculations and Lemma \ref{lemma:PropertiesX} we can write $\End_A(T)$ in matrix form and observe that it coincides with the definition of $B$.
The very last step is to verify that the algebra structures coincide, which is clear in all cases except for those multiplications which involve the quotient map onto $\cok(\theta_0)=W$. More precisely, one can check by direct computation that the diagram
\begin{align*}
\xymatrix@C=5pc
{
\Hom_A(e_i(1)A,X)\underset{\End_A(e_i(1)A)}{\otimes}\Hom_A(e_GA,e_i(1)A) \ar[r]^-{\mathsf{comp}} \ar[d]_\cong & \Hom_A(e_GA,X) \ar[dd]^\cong \\
\left(U_i^\vee\underset{D_i}{\otimes} D_i\right)\underset{D_i}{\otimes}V_i \ar[d]_\cong & \\
U_i^\vee\underset{D_i}{\otimes} V_i \ar@{->>}[r] & W
}
\end{align*}
is commutative for all $1\leq i\leq t$, where the map on top is given by composition and the map at the bottom is the cokernel map of $\theta_0$.
\end{proof}

\section{Derived equivalence via one-point-extensions}\label{section:CDalgebraViaOnePointExtension}

The goal of this section is to establish another perspective on the derived equivalence between the squid algebra and the Coxeter-Dynkin algebra. More precisely, we construct a different equivalence $\kD^b(A)\cong\kD^b(B)$ which does not come from a classical tilting module as in Section \ref{section:squidToCDviageneralisedAPR} but from a \emph{tilting complex}. The construction of this tilting complex makes use of reflection functors for hereditary algebras and the fact that all our algebras can be seen as one-point extensions of hereditary algebras.

Following \cite{RickardMoritaTheory} we call a bounded complex of projective modules $T\in\kK^b(\mathsf{proj-}A)$ a \emph{tilting complex} if the following two conditions are satisfied:
\begin{enumerate}
\item[(T1)] $\Hom_{\kK^b(\mathsf{proj-}A)}(T,T[i])=0$ for all $0\neq i\in \ZZ$ and
\item[(T2)] $\add(T)$ generates $\kK^b(\mathsf{proj-}A)$ as a triangulated category.
\end{enumerate}
Sometimes we will also use this terminology to denote a bounded complex of (not necessarily projective) $A$-modules if it is quasi-isomorphic to a complex of the above form.
It is proven in \cite[Theorem 6.4]{RickardMoritaTheory} that the existence of such a tilting complex $T$ with $\End(T)\cong B$ is equivalent to $\kK^b(\mathsf{proj-}A)$ and $\kK^b(\mathsf{proj-}B)$ being equivalent as triangulated categories. More concretely, in this case there is an equivalence of triangulated categories sending the tilting complex $T$ to the free module $B_B$ viewed as a complex concentrated in degree zero.
By \cite[Prop.~9.1]{RickardMoritaTheory} the above statements are equivalent to their dual versions using left modules.

It turns out that for our purposes it is more convenient to work with left modules for a moment. 
Let $A_0$ and $B_0$ be finite-dimensional algebras which are derived equivalent, i.\,e., we have equivalences of triangulated categories
\begin{align*}
\xymatrix@R=1pc{
& \kK^b(A_0\mathsf{-proj}) \ar[r]^\cong \ar@{^{(}->}[d] & \kK^b(B_0\mathsf{-proj}) \ar@{^{(}->}[d] \\
\qquad\Phi_0: & \kD^b(A_0\lmod) \ar[r]^\cong & \kD^b(B_0\lmod).
}
\end{align*}
Note that the vertical functors are in general fully faithful, and equivalences if and only if $A_0$ and $B_0$ are of finite global dimension. Then by Rickard's theorem there exists a tilting complex $T_0\in\kK^b(A_0\mathsf{-proj})$ with $\End(T_0)^\op\cong B_0$ such that $\Phi_0(T_0)={}_{B_0}(B_0)$. The following result -- proven by Barot and Lenzing in a slightly more restricted setting -- shows that the derived equivalence between $A_0$ and $B_0$ extends to a derived equivalence between certain one-point-extension algebras. For convenience, we give a full proof even though it works analogously to the one given in \cite{BarotLenzingOnePointExtension}.

\begin{theorem}\label{thm:BarotLenzing}
Let $G$ be a finite-dimensional division algebra and $N$ an $A_0$-$G$-bimodule such that $\Phi_0(N)$ is isomorphic in $\kD^b(B_0\lmod)$ to a complex concentrated in degree zero, in other words we can identify $\Phi_0(N)$ with a left $B_0$-module $N'$. Then $N'$ becomes a $B_0$-$G$-bimodule with a right $G$-action induced by that on $N$. Let
\begin{align*}
A=\begin{pmatrix} A_0 & N \\ 0 & G \end{pmatrix},\qquad 
B=\begin{pmatrix} B_0 & N' \\ 0 & G \end{pmatrix}
\end{align*}
be the corresponding one-point-extension algebras of $A_0$ resp.~$B_0$ and
\begin{align*}
e_0=\begin{pmatrix} 1_{A_0} & 0 \\ 0 & 0 \end{pmatrix},\quad 
e_G=\begin{pmatrix} 0 & 0 \\ 0 & 1_G \end{pmatrix}\quad\in A.
\end{align*}
Then $T=(Ae_0\otimes_{A_0}T_0)\oplus Ae_G[0]\in\kK^b(A\mathsf{-proj})$ is a tilting complex with $\End(T)^\op\cong B$.
\end{theorem}
\begin{proof}
The $A_0$-$G$-bimodule structure on $N$ gives rise to a morphism of $k$-algebras
\begin{align*}
\xymatrix{
G^\op \ar@{^{(}->}[r] \ar@{-->}[dr] & \End_{A_0\lmod}(N) \ar[r]^-\cong & \End_{\kD^b(A_0\lmod)}(N[0]) \ar[d]_\cong^{\Phi_0} \\
 & \End_{B_0\lmod}(N') \ar[r]^-\cong & \End_{\kD^b(B_0\lmod)}(N'[0]) \\
}
\end{align*}
which induces a morphism $G^\op\longrightarrow\End_{B_0\lmod}(N')$ and hence a $B_0$-$G$-bimodule structure on $N'$.\\
Recall that we have an adjoint pair
\begin{align}\label{eq:adjunctionOnePointExtension}
\xymatrix@C=3pc
{
A_0\lmod \ar@/^1pc/[rr]^{Ae_0\otimes_{A_0}-} & \perp & A\lmod, \ar@/^1pc/[ll]^{\Hom_{A\lmod}(Ae_0,-)}
}
\end{align}
where the left adjoint is fully faithful since the unit is an isomorphism. Moreover, the functor $Ae_0\otimes_{A_0}-$ restricts to a functor $A_0\mathsf{-proj}\longrightarrow A\mathsf{-proj}$, because its right adjoint is exact, and induces a fully-faithful functor
\begin{align*}
Ae_0\otimes_{A_0}-:\quad\kK^b(A_0\mathsf{-proj})\longrightarrow \kK^b(A\mathsf{-proj}).
\end{align*}
Consequently, we obtain
\begin{align*}
\Hom_{\kK^b(A\mathsf{-proj})}(Ae_0\otimes T_0,Ae_0\otimes T_0[n])
\cong \Hom_{\kK^b(A_0\mathsf{-proj})}(T_0,T_0[n])
\cong\left\{
\begin{array}{ll} B_0^\op, & \text{if } n=0,\\ 0, & \text{if } n\neq 0. \end{array}
\right.
\end{align*}
Since $Ae_G[0]\in\kK^b(A\mathsf{-proj})$ is a stalk complex, morphisms from $Ae_0\otimes T_0$ into it can be identified with the $0$-th cohomology of the complex $\Hom_{A\lmod}(Ae_0\otimes T_0, Ae_G)$. More generally, for any $n\in\ZZ$ we get
\begin{align*}
\Hom_{\kK^b(A\mathsf{-proj})}&(Ae_0\otimes T_0,Ae_G[n])
\cong \sH^{-n}(\Hom_{A\lmod}(Ae_0\otimes T_0, Ae_G))\\
&\stackrel{\eqref{eq:adjunctionOnePointExtension}}{\cong} \sH^{-n}(\Hom_{A\lmod}(T_0, \Hom_{A\lmod}(Ae_0,Ae_G))) \\
&\stackrel{\phantom{\eqref{eq:adjunctionOnePointExtension}}}{\cong} \sH^{-n}(\Hom_{A\lmod}(T_0, N))
\cong \Hom_{\kD^b(A_0\lmod)}(T_0,N[n])\\
&\stackrel{\phantom{\eqref{eq:adjunctionOnePointExtension}}}{\cong} \Hom_{\kD^b(B_0\lmod)}(\Phi_0(T_0),\Phi_0(N)[n])
\cong \Hom_{\kD^b(B_0\lmod)}({}_{B_0}B_0,N'[n]) \\
&\stackrel{\phantom{\eqref{eq:adjunctionOnePointExtension}}}{\cong} \Ext_{B_0}^n({}_{B_0}B_0,N')
\cong \left\{ \begin{array}{ll} N', & \text{if } n=0,\\ 0, & \text{if } n\neq 0. \end{array} \right.
\end{align*}
Another useful observation is that
\begin{align*}
\Hom_{A\lmod}(Ae_G,Ae_0\otimes_{A_0}X)\cong \underbrace{e_GAe_0}_{=0}\otimes_{A_0}X = 0
\qquad\forall X\in A_0\lmod,
\end{align*}
which implies that
\begin{align*}
\Hom_{\kK^b(A\mathsf{-proj})}&(Ae_G[0],Ae_0\otimes T_0[n]) = 0 \qquad\forall n\in \ZZ.
\end{align*}
Finally, it is clear that
\begin{align*}
\Hom_{\kK^b(A\mathsf{-proj})}&(Ae_G[0],Ae_G[n])
\cong \Ext_{A}^n(Ae_G,Ae_G)
\cong \left\{ \begin{array}{ll} G^\op, & \text{if } n=0,\\ 0, & \text{if } n\neq 0. \end{array} \right.
\end{align*}
This finishes the proof that $\Hom_{\kK^b(A\mathsf{-proj})}(T,T[n])=0$ whenever $n\neq 0$. We have also already shown that
\begin{align*}
\End_{\kK^b(A\mathsf{-proj})}(T)^\op\cong \begin{pmatrix} B_0^\op & 0 \\ N' & G^\op \end{pmatrix}^\op
\cong \begin{pmatrix} B_0 & N' \\ 0 & G \end{pmatrix}=B.
\end{align*}
It remains to verify the generation property for $T$. Let $\kT\subseteq\kK^b(A\mathsf{-proj})$ be the triangulated subcategory generated by $\add(T)$. Since $\add(T_0)$ generates $\kK^b(A_0\mathsf{-proj})$ as a triangulated category and the functor $Ae_0\otimes_{A_0}-$ maps $A_0\mathsf{-proj}$ to $\add(Ae_0)$, it follows that $\add(Ae_0)[0]$ is contained in $\kT$. As $Ae_G[0]$ is a direct summand of $T$, the subcategory $\kT$ contains $Ae_G[0]$ and hence $A\mathsf{-proj}[0]$. We can conclude that $\kT=\kK^b(A\mathsf{-proj})$ and $T$ is a tilting complex.
\end{proof}

\begin{remark}
As already mentioned, the statement of Theorem \ref{thm:BarotLenzing} is a bit stronger than the original formulation in \cite{BarotLenzingOnePointExtension}. The difference is that Barot and Lenzing stated it for the case $G=k$. In other words, the endomorphism algebra of the simple module corresponding to the extending vertex is the base field. However, it is harmless to assume that it is any finite-dimensional division algebra. We can even go one step further and leave the setting of one-point-extension algebras. If $G$ is any finite-dimensional $k$-algebra, then Theorem \ref{thm:BarotLenzing} with the same proof is still valid.
\end{remark}

We now come back to the study of the squid algebra and the Coxeter-Dynkin algebra and therefore use all the notation introduced in Section \ref{section:squidAndCDalgebra}.
Let $A_0$ and $B_0$ be the (hereditary) tensor algebras of the following $k$-species.
\begin{align*}
\xymatrix@R=1pc
{
D_1 \ar[r]^{D_1} & \dots & \dots \ar[r]^{D_1} & D_1 \ar[rdd]^{U_1} &  \\
D_2 \ar[r]^{D_2} & \dots & \dots \ar[r]^{D_2} & D_2 \ar[rd]_{U_2} &  \\
\vdots &  & \vdots &  & F \\
\vdots &  & \vdots &  &  \\
D_t \ar[r]^{D_t} & \dots & \dots \ar[r]^{D_t} & D_t \ar[ruu]_{U_t} &  \\
}
\qquad\qquad\qquad
\xymatrix@R=1pc
{
D_1 \ar[r]^{D_1} & \dots & \dots \ar[r]^{D_1} & D_1   \\
D_2 \ar[r]^{D_2} & \dots & \dots \ar[r] & D_2   \\
\vdots &  &  F \ar[ruu]^(.7){U_1^\vee} \ar[ru]_{U_2^\vee} \ar[rdd]_{U_t^\vee} &  \vdots   \\
\vdots &  &   & \vdots &   \\
D_t \ar[r]^{D_t} & \dots & \dots \ar[r]^{D_t} & D_t   \\
}
\end{align*}
In particular, these are subspecies of the underlying species of $A$ resp.~$B$ (where we just deleted one vertex each). One can see that the species on the right arises from the one on the left via a single reflection at the unique sink. Accordingly, we get an adjoint pair $(\SS^-,\SS^+)$ of reflection functors. 
For the general theory of reflection functors we refer to \cite[§2]{DlabRingel}. It is shown in \cite[Prop.~5.1]{APRTilting} that $\SS^+\cong \Hom_{B_0}({}_{B_0}X_{A_0},-)$ where $X\in B_0\lmod$ is the APR-tilting module
\begin{align*}
X=B_0(1-e_F)\oplus \tau^{-1}(B_0e_F)~~\in B_0\lmod
\end{align*}
whose opposite endomorphism algebra is isomorphic to $A_0$. We use the same notation for the primitive idempotents in $B_0$ corresponding to the respective vertices as in $A_0$. Note that $B_0e_F$ is a simple, projective left $B_0$-module, so that $X$ is indeed a tilting module. Let $A_0\lmod^*$ resp.~$B_0\lmod^*$ denote the subcategories of all modules without direct summands isomorphic to the simple module at the vertex labelled by the division algebra $F$. Then, the reflection functors $\SS^-$ and $\SS^+$ restrict to mutually inverse equivalences as depicted in the following diagram:
\begin{align}
\label{eq:reflectionFunctors}
\begin{minipage}{12cm}
\xymatrix@C=3pc@R=2pc
{
A_0\lmod \ar@/^/[rr]^{\SS^-} & \perp & B_0\lmod \ar@/^/[ll]^{\SS^+} \\
A_0\lmod^* \ar@/^/[rr] \ar@{_{(}->}[u] \ar@{^{(}->}[d] & \cong & B_0\lmod^* \ar@/^/[ll] \ar@{_{(}->}[u] \ar@{^{(}->}[d] \\
\kD^b(A_0\lmod) \ar@/^/[rr]^{\Phi_0} & \cong & \kD^b(B_0\lmod) \ar@/^/[ll] \\
}
\end{minipage}
\end{align}
A different perspective on the equivalence in the second row of \eqref{eq:reflectionFunctors} is provided by the theory of torsion pairs and the Brenner-Butler theorem (see e.\,g.~\cite[§VI]{AssemSimsonSkowronski}). The subcategory $B_0\lmod^*$ equals the torsion class
\begin{align*}
\kT(X)&=\left\{Y\in B_0\lmod~\vert~\Ext_{B_0}^1(X,Y)=0\right\}\\
&=\Gen(X)=\left\{Y\in B_0\lmod~\vert~\exists~X^{\oplus m}\twoheadrightarrow Y\right\}
\end{align*}
generated by the tilting module $X$. Since $\Ext_{B_0}^1(X,\mathcal{T}(X))=0$, one can show that the derived equivalence induced by $\SS^+$ (or by the isomorphic functor $\Hom_{B_0}(X,-)$) restricts to an equivalence between the torsion class $B_0\lmod^*$ in $B_0\lmod$ and a corresponding torsion free class $A_0\lmod^*$ in $A_0\lmod$. More precisely, every object in $\kT(X)$ has an $\add(X)$-resolution \cite[VI.2.6]{AssemSimsonSkowronski} which gets mapped to a projective resolution of an $A_0$-module under the functor $\Hom_{B_0}(X,-)$.
Either way, we can deduce that the equivalence $\Phi_0$ maps all $A_0$-modules in $A_0\lmod^*$, viewed as complexes concentrated in degree 0, to complexes concentrated in degree 0, namely to the $B_0$-modules in $B_0\lmod^*$. This is the key ingredient for the next result:
\begin{theorem}\label{thm:tiltSquidToOctopus1ptext}
Suppose that the conditions from Proposition \ref{Prop:OctopusCondition} are satisfied, then the complex
\begin{align*}
\sD(Ae_G)[0] ~\oplus~ \bop_{i=1}^t\bop_{j=1}^{p_i-1} \sD(Ae_i(j))[0] ~\oplus~ \tp(e_FA)[1]\quad\in \kD^b(\rmod A)
\end{align*}
is a tilting complex whose endomorphism algebra is isomorphic to $B$.
\end{theorem}
\begin{proof}
As we have already pointed out, the stalk complex $N[0]\in\kD^b(A_0\lmod)$ is mapped under $\Phi_0$ to a complex concentrated in degree 0 if and only if $N\in\ A_0\lmod^*$, meaning $N$ has no direct summand isomorphic to $\tp(A_0e_F)$. But this is one of the equivalent conditions from Proposition \ref{Prop:OctopusCondition}. Hence, Theorem \ref{thm:BarotLenzing} implies that
\begin{align*}
T=\Phi_0^{-1}({}_{B_0}B_0[0])\oplus Ae_G[0]\in\kD^b(A\lmod)
\end{align*}
is a tilting complex. It remains to identify $\Phi_0(N[0])$ with the left $B_0$-module $e_0Be_G$ and to compute the tilting complex $T_0=\Phi_0^{-1}\left({}_{B_0}B_0[0]\right)$ corresponding to the equivalence $\Phi_0$.\\
By definition of the reflection functor $\SS^-$, we have $e_i(j)(\SS^-(N))=e_i(j)N=V_i$ for all $1\leq i\leq t$ and $1\leq j\leq p_i-1$. Moreover, $e_F(\SS^-(N))$ fits into an exact sequence of left $F$-modules
\begin{align*}
\xymatrix{
& e_FN \ar[r] \ar@{=}[d] & \bop_{i=1}^t \Hom_{D_i}(U_i,D_i)\otimes_{D_i}Ne_i(j) \ar[r] \ar[d]^\cong & e_F(\SS^-(N)) \ar[r] \ar@{.>}[d]^{\cong} & 0 \\
0\ar[r] & M \ar[r]^-{\theta_0} & \bop_{i=1}^t U_i^\vee\otimes_{D_i}V_i \ar[r] & W \ar[r] & 0,
}
\end{align*}
showing that $e_F(\SS^-(N))\cong \cok(\theta_0)=W$. In regard of Definition \ref{Def:CoxeterDynkinAlgebra}, we can conclude that $\SS^-(N)\cong e_0Be_G$.\\
In order to compute $T_0$, we first note that $\Phi_0^{-1}$ sends $\add(X[0])$ to $A_0\mathsf{-proj}[0]$ and in particular $B_0(1-e_F)[0]$ to $A_0(1-e_F)[0]$. To see what happens to the remaining indecomposable direct summand $B_0e_F$, we compute its inverse Auslander-Reiten translate. We use that $\tau^{-1}(B_0e_F)\cong\mathsf{Tr}(\sD(B_0e_F))\cong\mathsf{Tr}(\tp(e_FB_0))$ and that
\begin{align}\label{eq:projPresTop(e_FB_0)}
\bop_{i=1}^tU_i^\vee\otimes_{D_i}e_i(1)B_0\stackrel{\pi}{\longrightarrow} e_FB_0 \longrightarrow \tp(e_FB_0)\longrightarrow 0
\end{align}
is a minimal projective presentation of the right $B_0$-module $\tp(e_FB_0)$ if we define $\pi$ as the morphism corresponding to $(\id_{U_i^\vee})_{1\leq i\leq t}$ under the isomorphism
\begin{align*}
\Hom_{B_0}\left(\bop_{i=1}^tU_i^\vee\otimes_{D_i}e_i(1)B_0, e_FB_0\right)
\cong \bop_{i=1}^t \Hom_{D_i}(U_i^\vee, \underbrace{\Hom_{B_0}(e_i(1)B_0,e_FB_0)}_{\cong U_i^\vee}).
\end{align*}
Applying the functor $\Hom_{B_0}(-,B_0)$ to \eqref{eq:projPresTop(e_FB_0)} yields
\begin{align*}
\xymatrix{
& B_0e_F \ar[r] \ar@{=}[dd] & \Hom_{B_0}\left(\bop_{i=1}^tU_i^\vee\otimes_{D_i}e_i(1)B_0, B_0\right) \ar[r] \ar[d]^\cong & \mathsf{Tr}(\tp(e_FB_0)) \ar[r] \ar@{-->}[dd]^\cong & 0 \\
& & \bop_{i=1}^t\Hom_{D_i}\left(U_i^\vee, \Hom_{B_0}(e_i(1)B_0,B_0)\right) \ar[d]^\cong & & \\
0 \ar[r] & B_0e_F \ar[r] & \bop_{i=1}^t B_0e_i(1)\otimes_{D_i}\underbrace{\left(U_i^\vee\right)^\vee}_{\cong U_i} \ar[r] & \tau^{-1}(B_0e_F) \ar[r] & 0.
}
\end{align*}
The bottom row is automatically a minimal projective presentation of the left $B_0$-module $\tau^{-1}(B_0e_F)$. Since $B_0$ is hereditary, it must be a short exact sequence and therefore also an $\add(X)$-coresolution of $B_0e_F$.
Now we can use $\Phi_0^{-1}(B_0e_i(1)[0])=A_0e_i(1)[0]$ and $\Phi_0^{-1}(\tau^{-1}(B_0e_F)[0])=A_0e_F[0]$ to deduce that the stalk complex $B_0e_F[0]$ gets mapped under $\Phi_0^{-1}$ to the complex
\begin{align}\label{eq:complexPhi^-1(B_0e_F)}
\bop_{i=1}^t A_0e_i(1)\otimes_{D_i}U_i \longrightarrow A_0e_F\qquad\in\kD^b(A_0\lmod)
\end{align}
concentrated in degrees 0 and 1. Note that the morphism in \eqref{eq:complexPhi^-1(B_0e_F)} is a monomorphism. More precisely, it is the radical inclusion of the projective module $A_0e_F$ so that its cokernel is the simple module $\tp(A_0e_F)$ and $\Phi_0^{-1}(B_0e_F[0])\cong \tp(A_0e_F)[-1]$. Summing up, we showed that
\begin{align*}
T~=~ \tp(Ae_F)[-1] ~\oplus~ \bop_{i=1}^t\bop_{j=1}^{p_i-1} Ae_i(j)[0] ~\oplus~ Ae_G[0] \quad\in \kD^b(A\lmod)
\end{align*}
is a tilting complex with $\End(T)^\op\cong B$. Since the algebra $A$ has finite global dimension, the contravariant equivalence
\begin{align*}
\kD^b(A\lmod)\cong\kK^b(A\mathsf{-proj})\stackrel{\sD}{\longrightarrow}\kK^b(\mathsf{inj-}A)\cong\kD^b(\rmod A)
\end{align*}
maps the tilting complex $T\in\kD^b(A\lmod)$ to the tilting complex
\begin{align*}
&\sD(T)~=~\sD(Ae_G)[0] ~\oplus~ \bop_{i=1}^t\bop_{j=1}^{p_i-1} \sD(Ae_i(j))[0] ~\oplus~ \tp(e_FA)[1]\quad\in \kD^b(\rmod A) \\
&\text{with}\qquad 
\End_{\kD^b(\rmod A)}(\sD(T))\cong \End_{\kD^b(A \lmod)}(T)^\op \cong B,
\end{align*}
which finishes the proof.
\end{proof}

\section{Canonical algebras}\label{section:canonicalAlgebras}

In this section, we introduce the most famous representative in the derived equivalence class of the squid algebra and the Coxeter-Dynkin algebra of canonical type, namely the canonical algebra as defined in \cite{Ringel}.
We start with the constructions which are required for the definition of a canonical algebra. Our choice of a tame bimodule, weights and exceptional points as in Section \ref{section:squidAndCDalgebra} is still fixed. We define $F$-$D_i$-bimodules $U_i^*$ and $D_i$-$G$-bimodules $V_i^+$ as follows:
\begin{align*}
U_i^* &:= \Hom_F\left({}_{D_i}(U_i)_F,{}_FF_F\right) \\
V_i^+ &:= \ker\left(U_i\otimes_FM\stackrel{\rho_i}{\longrightarrow} V_i\right)
\end{align*}
Note that the morphism $\rho$ is necessarily surjective for every regular representation $(U,V,\rho)$ of $M$, since otherwise there would exist a preinjective subrepresentation, namely $(U, \im(\rho), \rho)$. However, this is not possible since $\Hom(\kI,\kR)=0$. Therefore, we have a short exact sequence of $D_i$-$G$-bimodules
\begin{align}\label{eq:SESdefV+}
0\longrightarrow V_i^+\longrightarrow U_i\otimes_F M \stackrel{\rho_i}{\longrightarrow} V_i \longrightarrow 0.
\end{align}
For every $1\leq i \leq t$, we define a morphism $\tilde{\rho}_i$ of $F$-$G$-bimodules as the following composition:
\[
\xymatrix{
U_i^*\otimes_{D_i} V_i^+ \ar[rrrr]^-{\tilde{\rho}_i} \ar[dr]_{\id_{U_i^*}\otimes\incl} & & & & M \\
& U_i^* \otimes_{D_i}  U_i\otimes_F M \ar[rr]_-{\ev\otimes\id_M} & & F \otimes_F M \ar[ur]_-{\can}^-\cong
}
\]

\begin{definition}
The \emph{canonical algebra} associated to the tame bimodule ${}_FM_G$, weights $p_1,\dots,p_t$ and exceptional points $\rho_1,\dots,\rho_t$ is the finite-dimensional $k$-algebra
\begin{align*}
C := 
\left(
\begin{array}{c|ccc|c|ccc|c}
F & U_1^* & \cdots & U_1^* & \cdots & U_t^* & \cdots & U_t^* & M \\
\hline
 & D_1 & \cdots & D_1 & 0 & 0 & \cdots & 0 & V_1^+ \\
 &  & \ddots & \vdots & \vdots & \vdots &  & \vdots & \vdots \\
 &  &  & D_1 & 0 &  0 & \cdots & 0 & V_1^+ \\
\hline
 &  &  &  & \ddots & 0 & \cdots & 0 & \vdots \\
\hline
 &  &  &  &  & D_t & \cdots & D_t & V_t^+ \\
 &  &  &  &  &  & \ddots & \vdots & \vdots \\
 &  &  &  &  &  &  & D_t & V_t^+ \\
\hline
  &  &  &  &  &  &  &  & G 
\end{array}
\right),
\end{align*}
where $D_i$ appears precisely $p_i-1$ times on the diagonal. The multiplication in $A$ is given by the bimodule structures over the division algebras on the diagonal respectively by the $F$-$G$-bimodule homomorphisms $\tilde{\rho}_i:~U_i^*\otimes_{D_i} V_i^+\longrightarrow M$.
\end{definition}

\begin{remark}\label{rmk:CanonicalAlgebraSpecies}
The canonical algebra is a quotient of the tensor algebra of the $k$-species
\[
\xymatrix@R=1pc
{
 & D_1 \ar[r]^{D_1}  & D_1 \ar[r]^{D_1} & \dots & \dots \ar[r]^{D_1} & D_1 \ar[rdd]^{V_1^+} &   \\
 & D_2 \ar[r]^{D_2}  & D_2 \ar[r]^{D_2} & \dots & \dots \ar[r]^{D_2} & D_2 \ar[rd]_{V_2^+} &   \\
F \ar[uur]^{U_1^*}  \ar[ur]_{U_2^*}  \ar[ddr]_{U_t^*} \ar@/_1pc/[rrrrrr]_M 
 &   & \vdots &  & \vdots &  &  G \\
 &   & \vdots &  & \vdots &  &   \\
 & D_t \ar[r]^{D_t}  & D_t \ar[r]^{D_t} & \dots & \dots \ar[r]^{D_t} & D_t \ar[ruu]_{V_t^+} &   \\
}
\]
where the $i$-th arm has $p_i-1$ vertices. In fact the edge with the bimodule ${}_FM_G$ can be dropped in this species if and only if the map
\begin{align}
\left(\tilde{\rho}_1,\dots, \tilde{\rho}_t\right):\quad\bigoplus_{i=1}^t U_i^*\otimes_{D_i} V_i^+ \longrightarrow M
\end{align}
is surjective, which in turn is equivalent to the equivalent conditions from Proposition \ref{Prop:OctopusCondition}. This can be proven similarly as the equivalence of (1) and (6) in Proposition \ref{Prop:OctopusCondition} and has already been observed by Ringel \cite[§2.3]{Ringel}.
\end{remark}

\begin{example}
The algebra $C$ from Example \ref{example:complexNumbers} is precisely the canonical algebra associated to the respective input data. Hence, in this example the canonical algebra is a tame hereditary algebra. One can check that this always happens when we have equality in condition (6) of Proposition \ref{Prop:OctopusCondition}. 
\end{example}

In order to build the bridge from the squid to the canonical algebra, we adopt the strategy from Section \ref{section:CDalgebraViaOnePointExtension}. For this, let $A_0$ (as before) and $C_0$ be the tensor algebras of the following $k$-species:
\begin{align*}
\xymatrix@R=1pc
{
D_1 \ar[r]^{D_1} & \dots & \dots \ar[r]^{D_1} & D_1 \ar[rdd]^{U_1} &  \\
D_2 \ar[r]^{D_2} & \dots & \dots \ar[r]^{D_2} & D_2 \ar[rd]_{U_2} &  \\
\vdots &  & \vdots &  & F \\
\vdots &  & \vdots &  &  \\
D_t \ar[r]^{D_t} & \dots & \dots \ar[r]^{D_t} & D_t \ar[ruu]_{U_t} &  \\
}
\qquad\qquad
\xymatrix@R=1pc
{
 & D_1 \ar[r]^{D_1}  &  \dots & \dots \ar[r]^{D_1} & D_1  \\
 & D_2 \ar[r]^{D_2}  &  \dots & \dots \ar[r]^{D_2} & D_2   \\
F \ar[uur]^{U_1^*}  \ar[ur]_{U_2^*}  \ar[ddr]_{U_t^*}  
 &  \vdots &  & \vdots &  \\
 &  \vdots &  & \vdots &   \\
 & D_t \ar[r]^{D_t}  &  \dots & \dots \ar[r]^{D_t} & D_t   \\
}
\end{align*}
One can see that the right species arises from the left one by a sequence of $\sum_{i=1}^t\frac{(p_i-1)p_i}{2}$ reflections at source vertices. In particular, we obtain a corresponding composition of reflection functors $\SS^+:A_0\lmod\longrightarrow C_0\lmod$ and its left adjoint $\SS^-$.
Using inductively that reflection functors can be realized as $\Hom$-functors with APR-tilting modules \cite{APRTilting}, we can deduce that $\SS^+\cong \Hom_{A_0}(T_0,-)$ where
\begin{align*}
T_0~=~\bop_{i=1}^t\bop_{j=1}^{p_i-1} \tau^{-j}(A_0e_i(j))~\oplus~A_0e_F\qquad\in A_0\lmod
\end{align*}
is a tilting module with $\End_{A_0}(T_0)^\op\cong C_0$. If $\kT(T_0)=\Gen(T_0)$ denotes the torsion class in $A_0\lmod$ generated by $T_0$ and
\begin{align*}
\kY(T)=\left\{ Z\in C_0\lmod~\big\vert~\Tor_1^{C_0}(T_0,Z)=0 \right\}
\end{align*}
the corresponding torsion-free class in $C_0\lmod$, then we get the following commutative diagram:
\begin{align}
\label{eq:reflectionFunctorsSquidToCanonical}
\begin{minipage}{12cm}
\xymatrix@C=3pc@R=2pc
{
A_0\lmod \ar@/_/[rr]_{\SS^+} & \perp & C_0\lmod \ar@/_/[ll]_{\SS^-} \\
\kT(T_0) \ar@/_/[rr] \ar@{_{(}->}[u] \ar@{^{(}->}[d] & \cong & \kY(T_0) \ar@/_/[ll] \ar@{_{(}->}[u] \ar@{^{(}->}[d] \\
\kD^b(A_0\lmod) \ar@/_/[rr]_{\Phi_0} & \cong & \kD^b(C_0\lmod) \ar@/_/[ll] \\
}
\end{minipage}
\end{align}
By the Brenner-Butler theorem,
we have equivalences in the second row. More importantly for us, by this diagram all modules in $\kT(T_0)$ viewed as stalk complexes in degree 0 are mapped to stalk complexes in degree 0 under the equivalence $\Phi_0$. Hence, we arrive at the following result by Ringel \cite[§4]{Ringel} and can present it from a different perspective.

\begin{theorem}\label{thm:CanonicalTiltingModule}
The right $A$-module
\begin{align*}
T~=\sD(Ae_G)~\oplus~\bop_{i=1}^t\bop_{j=1}^{p_i-1} \tau^{j}\left(\sD(Ae_i(j))\right)~\oplus~\sD(Ae_F)
\end{align*}
viewed as a complex concentrated in degree 0 is a tilting complex in $\kD^b(\rmod A)$ with $\End(T)\cong C$. Moreover it satisfies the dual conditions of a classical tilting module.
\end{theorem}
\begin{proof}
Recall that the squid algebra $A$ can be written as the one-point-extension of $A_0$ along the $A_0$-$G$-bimodule $N$. The morphism $Ae_F\otimes_F M\longrightarrow Ae_G$ which corresponds to $\id_M$ under the isomorphism
\begin{align*}
\Hom_A(Ae_F\otimes_F M, Ae_G)\cong \Hom_F(M,\Hom_A(Ae_F,Ae_G)\cong \Hom_F(M,M)
\end{align*}
is an epimorphism onto the radical of $Ae_G$ since the morphisms $\rho_i$ are surjective. Hence, if we restrict it to $A_0$ (meaning if we apply $\Hom_A(Ae_0,-)$ to it) we obtain an epimorphism $A_0\otimes_F M\twoheadrightarrow N$ of left $A_0$-modules. Therefore
\begin{align*}
N\in \Gen(Ae_F) \subseteq \Gen(T_0) = \kT(T_0)
\end{align*}
which means that $\Phi_0(N[0])$ is isomorphic in $\kD^b(C_0\lmod)$ to a stalk complex concentrated in degree 0. Consequently, Theorem \ref{thm:BarotLenzing} implies that $Ae_0\otimes_{A_0}T_0[0]\oplus Ae_G[0]$ is a tilting complex in $\kD^b(A\lmod)$.
It is easy to see that  $Ae_0\otimes_{A_0}T_0\oplus Ae_G$ is even a classical (left) tilting module. It remains to show that the endomorphism algebra is indeed isomorphic to the canonical algebra. For this we need to compute $\SS^+(N)$ which is why we factor the reflection functor according to the following intermediate steps:

\begin{align*}
\begin{minipage}{6cm}
\scalebox{0.7}{
\xymatrix@R=1pc
{
D_1 \ar[r]^{D_1} & \dots & \dots \ar[r]^{D_1} & D_1 \ar[r]^{D_1} & D_1 \ar[rdd]^{U_1} &  \\
D_2 \ar[r]^{D_2} & \dots & \dots \ar[r]^{D_2} & D_2 \ar[r]^{D_2} & D_2 \ar[rd]_{U_2} &  \\
\vdots &  &  &  \vdots &  & F \\
\vdots &  &  &  \vdots &  &  \\
D_t \ar[r]^{D_t} & \dots & \dots \ar[r]^{D_t} & D_t \ar[r]^{D_t} & D_t \ar[ruu]_{U_t} &  \\
}
}
\end{minipage}
\quad
\stackrel{\SS^+_1}{\longrightarrow}
\quad
\begin{minipage}{6cm}
\scalebox{0.7}{
\xymatrix@R=1pc
{
D_1 \ar@{<-}[r]^{D_1} & \dots & \dots \ar@{<-}[r]^{D_1} & D_1 \ar@{<-}[r]^{D_1} & D_1 \ar[rdd]^{U_1} &  \\
D_2 \ar@{<-}[r]^{D_2} & \dots & \dots \ar@{<-}[r]^{D_2} & D_2 \ar@{<-}[r]^{D_2} & D_2 \ar[rd]_{U_2} &  \\
\vdots &  &  &  \vdots &  & F \\
\vdots &  &  &  \vdots &  &  \\
D_t \ar@{<-}[r]^{D_t} & \dots & \dots \ar@{<-}[r]^{D_t} & D_t \ar@{<-}[r]^{D_t} & D_t \ar[ruu]_{U_t} &  \\
}
}
\end{minipage}
\qquad\qquad
\end{align*}
\begin{align*}
\quad
\stackrel{\SS^+_2}{\longrightarrow}
\quad
\begin{minipage}{6cm}
\scalebox{0.7}{
\xymatrix@R=1pc
{
D_1 \ar@{<-}[r]^{D_1} & \dots & \dots \ar@{<-}[r]^{D_1} & D_1 \ar[r]^{D_1} & D_1 \ar@{<-}[rdd]^{U_1^*} &  \\
D_2 \ar@{<-}[r]^{D_2} & \dots & \dots \ar@{<-}[r]^{D_2} & D_2 \ar[r]^{D_2} & D_2 \ar@{<-}[rd]_{U_2^*} &  \\
\vdots &  &  &  \vdots &  & F \\
\vdots &  &  &  \vdots &  &  \\
D_t \ar@{<-}[r]^{D_t} & \dots & \dots \ar@{<-}[r]^{D_t} & D_t \ar[r]^{D_t} & D_t \ar@{<-}[ruu]_{U_t^*} &  \\
}
}
\end{minipage}
\quad
\stackrel{\SS^+_3}{\longrightarrow}
\quad
\begin{minipage}{6cm}
\scalebox{0.7}{
\xymatrix@R=1pc
{
D_1 \ar@{<-}[r]^{D_1} & \dots & \dots \ar@{<-}[r]^{D_1} & D_1 \ar@{<-}[r]^{D_1} & D_1 \ar@{<-}[rdd]^{U_1^*} &  \\
D_2 \ar@{<-}[r]^{D_2} & \dots & \dots \ar@{<-}[r]^{D_2} & D_2 \ar@{<-}[r]^{D_2} & D_2 \ar@{<-}[rd]_{U_2^*} &  \\
\vdots &  &  &  \vdots &  & F \\
\vdots &  &  &  \vdots &  &  \\
D_t \ar@{<-}[r]^{D_t} & \dots & \dots \ar@{<-}[r]^{D_t} & D_t \ar@{<-}[r]^{D_t} & D_t \ar@{<-}[ruu]_{U_t^*} &  \\
}
}
\end{minipage}
\end{align*}
That means that $\SS^+=\SS^+_3\SS^+_2\SS^+_1$, where $\SS^+_1$ is the product of all reflection functors corresponding to vertices inside the arms until the rightmost vertices in the arms are the only source vertices.
Then $\SS^+_2$ is the product of the (commuting) reflection functors associated to these $t$ source vertices. Finally $\SS^+_3$ is the product of all reflection functors at source vertices inside the arms until all arrows are pointing to the left.
Using this, we will compute the image of $N$ under $\SS^+$ step by step. For a left $A_0$-module $X$ and an idempotent $e\in A_0$ the following table shows the left $eA_0e$-module $eX$, where $1\leq i\leq t$ and $2\leq j\leq p_i-1$:
\begin{align*}
\begin{array}{|l|c|c|c|c|}
\hline
 & ~N~ & \SS^+_1N & \SS^+_2\SS^+_1N & \SS^+_3\SS^+_2\SS^+_1N \\[+1pt] \hline
e_F & M & M & M & M \\[+1pt] \hline
e_i(1) & V_i & V_i & V_i^+ & V_i^+ \\[+1pt] \hline
e_i(j) & V_i & 0 & 0 & V_i^+ \\[+1pt] \hline
\end{array}
\end{align*}
It is clear that the second row stays unchanged since $\SS^+$ contains no reflection at the rightmost vertex. Moreover, the $\SS^+_1$-action as well as the $\SS^+_3$-action can be easily deduced from the definition of reflection functors. The most interesting part is when $\SS^+_2$ comes into play. The rightmost vertex in the $i$-th arm is a source vertex with two arrows pointing away. The associated maps in $\SS^+_1N$, viewed as a (left) species representation are
\begin{align*}
U_i\otimes_F \overbrace{e_F(\SS^+_1N)}^{=M}&\longrightarrow e_i(1)(\SS^+_1N),\\
D_i\otimes_{D_i} \underbrace{e_i(2)(\SS^+_1N)}_{=0}&\longrightarrow e_i(1)(\SS^+_1N)
\end{align*}
and the left $D_i$-module $e_i(1)(\SS^+_2\SS^+_1N)$ is the kernel in the exact sequence
\begin{align*}
0\longrightarrow e_i(1)(\SS^+_2\SS^+_1N)\longrightarrow
\underbrace{U_i\otimes_F e_F(\SS^+_1N) \,\oplus\, D_i\otimes_{D_i} e_i(2)(\SS^+_1N)}_{=U_i\otimes M}
\longrightarrow \underbrace{e_i(1)(\SS^+_1N)}_{=V_i}.
\end{align*}
We can therefore conclude that $e_i(1)(\SS^+_2\SS^+_1N)\cong V_i^+$. Hence, we have shown that $\SS^+N$ coincides with the left $C_0$-module $e_0Ce_G$ at every vertex. One can verify by direct computation that they are even isomorphic as $C_0$-modules. According to Theorem \ref{thm:BarotLenzing}, the classical left tilting module
\begin{align*}
Ae_0\otimes_{A_0}T_0\oplus Ae_G \qquad \in A_0\lmod
\end{align*}
has an endomorphism algebra isomorphic to
\begin{align*}
\begin{pmatrix}
C_0 & \SS^+(N) \\
0 & G
\end{pmatrix}^\op
\cong C^\op.
\end{align*}
Applying the $k$-dual $\sD=\Hom_k(-,k)$ to this, leads to the right $A$-module from the corollary, which then has all the properties as claimed.
\end{proof}

\section{Tubular symbol and Weyl group}\label{section:SymbolWeylGroup}

In this section, we want to introduce the tubular symbol of a squid algebra, Coxeter-Dynkin algebra resp.~canonical algebra and build the bridge to Lenzing's work on canonical bilinear lattices \cite{LenzingKTheoreticStudy}.
Moreover we present the connection between Coxeter-Dynkin algebras and Saito's work on marked extended affine root systems \cite{SaitoI}.
We start by recalling the central definitions from \cite{LenzingKTheoreticStudy}.
\begin{definition}
A \emph{bilinear lattice} is a free abelian group of finite rank $V$ equipped with a non-degenerate bilinear form such that there exists an isomorphism $\tau:\,V\longrightarrow V$ (called \emph{Coxeter transformation}) satisfying
\begin{align*}
\langle y,x\rangle = -\langle x,\tau y\rangle\qquad\quad \forall\, x,y\in V.
\end{align*}
The \emph{radical} of a bilinear lattice $V$ is the subgroup
\begin{align*}
\rad(V) &= \left\{y\in V~\vert~\forall x\in V:~\langle y,x\rangle = -\langle x,y \rangle~\right\}\\
&= \left\{y\in V~\vert~\tau(y)=y\right\},
\end{align*}
in other words it is the left and right kernel of the \emph{symmetrised bilinear form} (which is defined by $(x,y)=\langle x,y\rangle +\langle y,x\rangle$ for all $x,y\in V$).
\end{definition}
One can show that in this case, the Coxeter transformation $\tau$ is uniquely determined by the bilinear form.
As an example, the Grothendieck group of any finite-dimensional $k$-algebra of finite global dimension equipped with the Euler form is a bilinear lattice. In this case the Coxeter transformation is given by the action of the derived Auslander-Reiten translation.

\begin{defProp}\label{defProp:CBL}
Given integers $t\geq 1$ and $p_i\geq 2$, $e_i\geq 1$, $f_i\geq 1$ for $1\leq i\leq t$ and $\varepsilon\in\{1,2\}$, we set
\begin{align}\label{eq:d_ikappa}
d_i=e_if_i\qquad\text{and}\qquad\kappa= \lcm\left\{ \frac{e_i}{\gcd(e_i,\varepsilon f_i)}~\bigg\vert~1\leq i\leq t. \right\}.
\end{align}
Let $V$ be the free abelian group with basis
\begin{align*}
a,~w,~ s_i^{(j)}\quad \left(\begin{array}{l} 1\leq i\leq t \\ 0\leq j\leq p_i-2\end{array}\right).
\end{align*}
We equip $V$ with the bilinear form such that the non-vanishing values on the basis elements are
\begin{align*}
\begin{array}{rlrl}
\langle a,a\rangle &= \kappa, & \langle a,s_i^{(0)}\rangle &= \kappa\varepsilon f_i \\[+5pt]
\langle a,w\rangle &= \kappa\varepsilon = -\langle w,a\rangle, & \quad\langle s_i^{(j)},s_i^{(j)}\rangle &= \frac{\kappa\varepsilon f_i}{e_i} = -\langle s_i^{(j-1)},s_i^{(j)}\rangle
\end{array}
\end{align*}
Then, $V$ is a bilinear lattice with $\tau^{j}s_i^{(0)}=s_i^{(j)}$ and $\tau^{p_i}s_i^{(0)}=s_i^{(0)}$ for all $1\leq i\leq t$ and $0\leq j\leq p_i-2$. We will write $s_i=s_i^{(0)}$ for simplicity. Moreover, $V$ is called \emph{canonical bilinear lattice} corresponding to the \emph{symbol}
\begin{align}\label{eq:tubularSymbol}
\sigma = \left(\begin{array}{c|c}
p_1\, p_2\,\dots \, p_t & \\
d_1\, d_2\,\dots \, d_t & \varepsilon\\
f_1\, f_2\,\dots \, f_t & \\
\end{array}\right).
\end{align}
\end{defProp}

Another useful basis for the identification of $V$ with the Grothendieck group of a canonical algebra is the so-called \emph{canonical basis} \cite[§10]{LenzingKTheoreticStudy}
\begin{align*}
a,~ a_i(j) = \sum_{l=1}^j\tau^{-l}s_i+\varepsilon f_ia
\quad \left(\begin{array}{l} 1\leq i\leq t \\ 0\leq j\leq p_i-2\end{array}\right),~
w+\varepsilon a.
\end{align*}
With respect to the canonical basis, the bilinear form on $V$ is given by the matrix
\begin{align}\label{eq:MatrixCanonicalBasis}
\left(
\begin{array}{c|ccc|c|ccc|c}
\kappa & \kappa\varepsilon f_1 & \cdots &\kappa\varepsilon f_1 & \cdots & \kappa\varepsilon f_t & \cdots & \kappa\varepsilon f_t & 2\kappa\varepsilon \\
\hline
 & \frac{\kappa\varepsilon f_1}{e_1} & \cdots & \frac{\kappa\varepsilon f_1}{e_1} & 0 & 0 & \cdots & 0 & \kappa\varepsilon^2 f_1 \\
 &  & \ddots & \vdots & \vdots & \vdots &  & \vdots & \vdots \\
 &  &  & \frac{\kappa\varepsilon f_1}{e_1} & 0 &  0 & \cdots & 0 & \kappa\varepsilon^2 f_1 \\
\hline
 &  &  &  & \ddots & 0 & \cdots & 0 & \vdots \\
\hline
 &  &  &  &  & \frac{\kappa\varepsilon f_t}{e_t} & \cdots & \frac{\kappa\varepsilon f_t}{e_t} & \kappa\varepsilon^2 f_t \\
 &  &  &  &  &  & \ddots & \vdots & \vdots \\
 &  &  &  &  &  &  & \frac{\kappa\varepsilon f_t}{e_t} & \kappa\varepsilon^2 f_t \\
\hline
  &  &  &  &  &  &  &  & \kappa\varepsilon^2 
\end{array}
\right).
\end{align}
To identify the Grothendieck group of a canonical algebra with a canonical bilinear lattice, we have to know which parameters we have to choose in Definition \ref{defProp:CBL} according to the input data for the canonical algebra. We fix a tame bimodule ${}_FM_G$, weights $p_1,\dots,p_t\geq 2$ and exceptional points $\rho_i$ and let $C$ be the corresponding canonical algebra. Let us assume for a moment that $\varepsilon$ as defined in \eqref{eq:Epsilon} equals either $1$ or $2$. We define the numerical invariants
\begin{align}
\nonumber
e_i&=\dim {}_{D_i}U_i = \frac{1}{\varepsilon}\dim {}_{D_i}V_i,\\ \label{eq:efdDefStandard}
f_i&=\dim(V_i)_G = \frac{1}{\varepsilon}\dim (U_i)_F,\\ \nonumber
d_i&=e_if_i,\phantom{\frac{1}{\varepsilon}}
\end{align}
where the equalities in the first two lines follow from \eqref{eq:FGdimensions} and \eqref{eq:PRIdecompositionDimVector}.

\begin{lemma}\label{lemma:K_0(C)isCBL}
The Grothendieck group $\sK_0(C\lmod)$ equipped with the Euler form
\begin{align*}
\langle -,-\rangle:~\sK_0(C\lmod)\times\sK_0(C\lmod)&\longrightarrow \ZZ \\
([X],[Y])&\longmapsto \sum_{m=0}^\infty \dim_k\Ext_C^m(X,Y)
\end{align*}
is - up to a constant factor - isomorphic to a canonical bilinear lattice with symbol \eqref{eq:tubularSymbol}.
\end{lemma}
\begin{proof}
First, we note that $\kappa$ defined as in \eqref{eq:d_ikappa} divides $\dim_kF$. For this, it suffices to show that
\begin{align*}
\frac{e_i}{\gcd(e_i,~\varepsilon f_i)}~~\bigg\vert~~ \dim_kF\qquad\qquad & \forall~1\leq i\leq t\\[+5pt]
\Leftrightarrow \qquad e_i~~\big\vert~~\gcd(e_i\dim_kF,~\varepsilon f_i \dim_kF) \qquad\qquad & \forall~1\leq i\leq t,
\end{align*}
which holds true since
\begin{align*}
\varepsilon f_i \dim_kF = \dim (U_i)_F \cdot \dim_kF = \dim_k U_i =  \dim {}_{D_i}U_i \cdot\dim_kD_i = e_i\dim_kD_i
\end{align*}
is a multiple of $e_i$. Furthermore, we compute
\begin{align*}
\dim_kU_i^*&=\dim_kU_i = \dim (U_i)_F\cdot\dim_kF = \varepsilon f_i\dim_kF\\
\dim_kM &= \dim{}_FM\cdot\dim_kF = 2\varepsilon\dim_kF
\end{align*}
\begin{align*}
\dim_kD_i &= \frac{\dim_kU_i}{\dim{}_{D_i}U_i} = \frac{\varepsilon f_i}{e_i}\dim_kF\\
\dim_kV_i^+&=\dim{}_{D_i}V_i^+\cdot\dim_kD_i = \dim{}_{D_i}V_i\cdot\dim_kD_i\\
&=\varepsilon e_i \frac{\varepsilon f_i}{e_i}\dim_kF = \varepsilon^2 f_i \dim_kF\\
\dim_kG &= \varepsilon^2\dim_kF.
\end{align*}
Combining this with the definition of $C$, we see that the Matrix corresponding to the Euler form on $\sK_0(C\lmod)$ in the basis of the classes of projective $C$-modules is a scalar multiple of \eqref{eq:MatrixCanonicalBasis}. Therefore $\sK_0(C\lmod)$ is isomorphic to $V$, up to a constant factor.
\end{proof}

Due to the previous lemma it makes sense to call \eqref{eq:tubularSymbol} the symbol of the canonical algebra $C$ respectively of the squid algebra $A$ or the Coxeter-Dynkin algebra of canonical type $B$. The next goal is to get rid of the constraint $\varepsilon\in\{1,2\}$. It has already been shown by Ringel that the opposite algebra of a canonical algebra is again a canonical algebra for (possibly) different input data \cite[Prop.~1]{Ringel}.

\begin{lemma}
The Grothendieck groups of $C\lmod$ and $\rmod C\cong C^\op\lmod$ are isomorphic as bilinear lattices. In particular, they are isomorphic - up to a constant factor - to a canonical bilinear lattice of the same symbol.\\
Moreover, suppose that $\varepsilon=\frac{1}{2}$, then $\sK_0(C\lmod)$ is a canonical bilinear lattice of symbol
\begin{align*}
\sigma = \left(\begin{array}{c|c}
p_1\, p_2\,\dots \, p_t & \\
d_1\, d_2\,\dots \, d_t & 2\\
f_1\, f_2\,\dots \, f_t & \\
\end{array}\right)
\end{align*}
with
\begin{align}
\nonumber
e_i&=\varepsilon\dim {}_{D_i}U_i = \dim {}_{D_i}V_i,\\ \label{eq:efdDefOpposite}
f_i&=\varepsilon\dim(V_i)_G = \dim (U_i)_F,\\ \nonumber
d_i&=e_if_i.\phantom{\frac{1}{\varepsilon}}
\end{align}
\end{lemma}
\begin{proof}
We introduce another basis
\begin{align}\label{eq:oppositeCanonicalBasis}
\varepsilon a-w,~ b_i(j) = \sum_{l=1}^j\tau^{-l}s_i+ f_i(\varepsilon a-w)
\quad \left(\begin{array}{l} 1\leq i\leq t \\ 0\leq j\leq p_i-2\end{array}\right),~
a,
\end{align}
of a bilinear canonical lattice of symbol \eqref{eq:tubularSymbol}. One can check by direct computation that this is indeed a basis of $V$ and that the bilinear form can be described in this basis by the matrix
\begin{align}\label{eq:MatrixCanonicalBasis2ndType}
\left(
\begin{array}{c|ccc|c|ccc|c}
\kappa\varepsilon^2 & \kappa\varepsilon^2 f_1 & \cdots &\kappa\varepsilon^2 f_1 & \cdots & \kappa\varepsilon^2 f_t & \cdots & \kappa\varepsilon^2 f_t & 2\kappa\varepsilon \\
\hline
 & \frac{\kappa\varepsilon f_1}{e_1} & \cdots & \frac{\kappa\varepsilon f_1}{e_1} & 0 & 0 & \cdots & 0 & \kappa\varepsilon f_1 \\
 &  & \ddots & \vdots & \vdots & \vdots &  & \vdots & \vdots \\
 &  &  & \frac{\kappa\varepsilon f_1}{e_1} & 0 &  0 & \cdots & 0 & \kappa\varepsilon f_1 \\
\hline
 &  &  &  & \ddots & 0 & \cdots & 0 & \vdots \\
\hline
 &  &  &  &  & \frac{\kappa\varepsilon f_t}{e_t} & \cdots & \frac{\kappa\varepsilon f_t}{e_t} & \kappa\varepsilon f_t \\
 &  &  &  &  &  & \ddots & \vdots & \vdots \\
 &  &  &  &  &  &  & \frac{\kappa\varepsilon f_t}{e_t} & \kappa\varepsilon f_t \\
\hline
  &  &  &  &  &  &  &  & \kappa 
\end{array}
\right).
\end{align}
The opposite algebra of the canonical algebra can be written as
\begin{align}
\label{eq:canonicalAlgebraOpposite}
C^\op \cong 
\left(
\begin{array}{c|ccc|c|ccc|c}
G^\op & V_1^+ & \cdots & V_1^+ & \cdots & V_t^+ & \cdots & V_t^+ & M \\
\hline
 & D_1^\op & \cdots & D_1^\op & 0 & 0 & \cdots & 0 & U_1^* \\
 &  & \ddots & \vdots & \vdots & \vdots &  & \vdots & \vdots \\
 &  &  & D_1^\op & 0 &  0 & \cdots & 0 & U_1^* \\
\hline
 &  &  &  & \ddots & 0 & \cdots & 0 & \vdots \\
\hline
 &  &  &  &  & D_t^\op & \cdots & D_t^\op & U_t^* \\
 &  &  &  &  &  & \ddots & \vdots & \vdots \\
 &  &  &  &  &  &  & D_t^\op & U_t^* \\
\hline
  &  &  &  &  &  &  &  & F^\op 
\end{array}
\right).
\end{align}
Therefore, the basis \eqref{eq:oppositeCanonicalBasis} corresponds to the classes of indecomposable projective modules in $\sK_0(C^\op\lmod)$. This shows that $\sK_0(C\lmod)\cong\sK_0(C^\op\lmod)$.\\
Now suppose that $\varepsilon=\frac{1}{2}$, then the opposite algebra $C^\op$, which is also a canonical algebra, satisfies $\varepsilon=2$ and we can therefore find the parameters $e_i$ and $f_i$ by comparing \eqref{eq:MatrixCanonicalBasis} and \eqref{eq:canonicalAlgebraOpposite}. More precisely, one can easily see that if we define the parameters as in \eqref{eq:efdDefOpposite}, then taking $k$-dimensions in \eqref{eq:canonicalAlgebraOpposite} yields the matrix \eqref{eq:MatrixCanonicalBasis}.
\end{proof}

\begin{remark}
Note that we have used left modules, since this was more convenient for identifying $\sK_0(C\lmod)$ with a canonical bilinear lattice. By the last theorem we know that passing to right modules (in other words passing to the opposite algebra) does not change the bilinear lattice up to isomorphism. Regarding the squid algebra, it is clear that its opposite algebra is not a squid algebra, but it is derived equivalent to an ordinary squid algebra with (possibly) different input data. The same holds true for Coxeter-Dynkin algebras of canonical type. When extracting the symbol from the minimal model, the weights and the exceptional points, one has to keep in mind that \eqref{eq:efdDefStandard} gives the correct parameters for $\varepsilon=2$ and \eqref{eq:efdDefOpposite} gives the correct parameters for $\varepsilon=\frac{1}{2}$ while for $\varepsilon=1$ the two expressions coincide.
\end{remark}

More generally, it is shown in \cite[Cor.~4.3]{LenzingdelaPena} that the Grothendieck group of any connected finite-dimensional algebra with a so-called \emph{sincere separating exact subcategory} is a canonical bilinear lattice.
From \cite{LenzingKTheoreticStudy} we know that several lattice-theoretic properties of a canonical bilinear lattice depend only on (the sign of) the number
\begin{align*}
\delta=\delta[V] = p\cdot\left(\sum_{i=1}^t{d_i\left(1-\frac{1}{p_i}\right)-\frac{2}{\varepsilon}}\right)
\end{align*}
with $p=\lcm\{p_1,\dots,p_t\}$ which can be computed from the symbol. There are also representation-theoretic analogues involving the representation type of the algebras \cite[Thm.~7.1]{LenzingdelaPena}\cite[§2.7]{Ringel} which we summarise in the table at the end of section \ref{section:RepresentationType}. The canonical bilinear lattice (as well as the squid algebra/Coxeter-Dynkin algebra/canonical algebra) is called \emph{domestic} (resp.~\emph{tubular}, resp.~\emph{wild}) if $\delta<0$ (resp.~$\delta=0$, resp.~$\delta>0$). A list of all domestic and tubular symbols can be found in \cite{LenzingKTheoreticStudy}.

In the following, we will again focus on the Coxeter-Dynkin algebra of canonical type $B$ associated to the same initial data. We write
\begin{align*}
e_F,~ e_1(p_1-1),\dots,e_1(1),\dots\quad\dots,e_t(p_t-1),\dots,e_t(1),~e_G
\end{align*}
for the primitive idempotents in $B$ and
\begin{align*}
S_F=\tp(e_FB),\quad S_i(j)=\tp(e_i(j)B),\quad S_G=\tp(e_GB)=e_GB
\end{align*}
for the corresponding simple right $B$-modules.

\begin{lemma}\label{lemma:EulerFormOctopusSimples}
The non-vanishing values of the Euler form on $\sK_0(\rmod B)$ in the basis of the classes of simple $B$-modules are
\begin{align*}
\setlength{\arraycolsep}{2pt} 
\begin{array}{rclrcl}
\langle [S_F],[S_F]\rangle &=& \dim_k(F),&
\langle [S_F],[S_i(1)]\rangle &=& -\varepsilon f_i \cdot\dim_k(F),\\[+5pt]
\langle [S_F],[S_G]\rangle &=& 2\varepsilon\cdot\dim_k(F),&
\langle [S_i(j)],[S_i(j)]\rangle &=& \frac{\varepsilon f_i}{e_i}\cdot\dim_k(F),\\[+5pt]
\langle [S_i(j+1)],[S_i(j)]\rangle &=& -\frac{\varepsilon f_i}{e_i}\cdot\dim_k(F) & \quad(1\leq j \leq p_i-2),&&\\[+5pt]
\langle [S_i(1)],[S_G]\rangle &=& -\varepsilon^2 f_i \cdot\dim_k(F)&
\langle [S_G],[S_G]\rangle &=& \varepsilon^2\cdot\dim_k(F),
\end{array}
\end{align*}
where $1\leq i\leq t$ and $1\leq j\leq p_i-1$.
\end{lemma}
\begin{proof}
Most of the equalities immediately follow from the computations in the proof of Lemma \ref{lemma:K_0(C)isCBL}. More precisely, we can use that if two vertices are connected by an edge in the species \eqref{eq:speciesOctopus}, the group $\Ext^1$ between the corresponding simple objects is isomorphic to the bimodule on the edge. The only non-trivial second $\Ext$-group between simple $B$-modules is $\Ext_B^2(S_F,S_G)$ which we compute explicitly hereinafter.
We claim that a projective resolution of $S_F$ is given by
\begin{align}\label{eq:projResS_F}
0\longrightarrow
{}_FM_G\otimes_G e_GB \longrightarrow
\bop_{i=1}^tU_i^\vee \otimes_{D_i}e_i(1)B \stackrel{\pi}{\longrightarrow}
e_FB \longrightarrow
S_F\longrightarrow 0
\end{align}
where $\pi$ corresponds to $(\id_{U_i^\vee})_{1\leq i\leq t}$ under the isomorphism
\begin{align*}
\Hom_B\left( \bop_{i=1}^tU_i^\vee \otimes_{D_i}e_i(1)B~,~ e_FB \right)
\cong \bop_{i=1}^t\Hom_{D_i}(U_i^\vee ~,~ \underbrace{\Hom_B(e_i(1)B,e_FB)}_{\cong U_i^\vee}).
\end{align*}
In order to compute the kernel of $\pi$, we first observe that it is injective at every vertex except at the sink vertex. Hence, $\ker(\pi)$ is supported at the vertex labelled by the division algebra $G$. Restricted to this vertex, the morphism $\pi$ takes the form
\begin{align*}
0\longrightarrow {}_FM_G \stackrel{\theta_0}{\longrightarrow} 
\bop_{i=1}^tU_i^\vee \otimes_{D_i}V_i \stackrel{\pi_0}{\longrightarrow}
{}_FW_G\longrightarrow 0
\end{align*}
which shows that \eqref{eq:projResS_F} is indeed a projective resolution of $S_F$. Applying $\Hom_B(-,S_G)$ to this, yields a stalk complex with the only non-vanishing homology in degree 2, namely
\begin{align*}
\Ext_B^2(S_F,S_G)\cong \Hom_B\left({}_FM_G\otimes_G e_GB, e_GB\right)\cong {}_FM_G.
\end{align*}
This shows that $\langle [S_F],[S_G]\rangle = 2\varepsilon\cdot\dim_k(F)$ and finishes the proof.
\end{proof}

\begin{corollary}\label{cor:radical}
Let $w\in\sK_0(\rmod B)$ be the element
\begin{align*}
w=
\begin{cases}
\varepsilon[S_F]-[S_G], & \text{ if }~~ \varepsilon\in\{1,2\}, \\[+1pt]
[S_F]-2[S_G], & \text{ if }~~ \varepsilon=\frac{1}{2} .
\end{cases}
\end{align*}
Then $\ZZ w$ is a direct summand of the radical of $\sK_0(\rmod B)$. Moreover, if $B$ is domestic or wild, then $\rad(\sK_0(\rmod B))=\ZZ w$.
\end{corollary}
\begin{proof}
Using Lemma \ref{lemma:EulerFormOctopusSimples}, a direct computation shows that $w$ is contained in the radical. Since the classes of simple $B$-modules form a basis of $\sK_0(\rmod B)$, it is also clear that $\ZZ w$ is a direct summand. Moreover, we know that the radical of $\sK_0(\rmod B)$ has rank $1$ if $B$ is of domestic or wild type \cite[Prop.~10.3]{LenzingKTheoreticStudy}.
\end{proof}

Building on the previous calculations, we want to change our perspective to describe a link to geometric group theory. We can view finite-dimensional algebras as a source of bilinear forms, looking at the associated (symmetrised) Euler form. Given such a symmetric bilinear form, one can introduce the notion of \emph{(generalised) root systems} and study the corresponding reflection groups (or \emph{Weyl groups}). Since domestic canonical algebras are derived equivalent to tame hereditary algebras \cite[§2.7]{Ringel}, the Weyl groups are affine Coxeter groups. In the case $\delta=0$, the symmetrised Euler form is still positive semidefinite, but has a $2$-dimensional radical. The associated Weyl groups belong to a different interesting class of reflection groups, the so-called \emph{extended affine Weyl group} (or \emph{elliptic Weyl group}). In particular, Saito studied \emph{extended affine root systems} and classified them via Dynkin diagrams. The Dynkin diagrams that show up for the extended affine root systems of codimension $1$ \cite[Table 1]{SaitoI} look very similar to the underlying weighted graphs \eqref{eq:speciesOctopus} of tubular Coxeter-Dynkin algebras, with dotted edges in Saito's Dynkin diagrams indicating relations in the algebra. Comparing \cite[§8.2]{SaitoI} with Lemma \ref{lemma:EulerFormOctopusSimples}, this connection becomes even clearer, namely the basis consisting of the classes of simple $B$-modules corresponds to the basis of simple roots associated to the vertices of Saito's Dynkin diagrams.

\section{Representation type}\label{section:RepresentationType}

Our next goal is to study the representation type of Coxeter-Dynkin algebras of canonical type. It is clearly of infinite (resp.~wild) representation type for $\delta\geq 0$ (resp.~$\delta>0)$ since it is a one-point-extension of a hereditary algebra of infinite (resp.~wild) representation type. In the domestic case, we have already seen in Example \ref{example:complexNumbers} that $B$ can be representation-finite. Hereinafter we will show that this is a general phenomenon. Since a domestic Coxeter-Dynkin algebra is always derived equivalent to a tame hereditary algebra, the following result of Ringel about tilted algebras will turn out to be useful.

\begin{theorem}[{\cite[§4.2(8)]{RingelTameAlgebras}}] \label{thm:RingelRepFiniteTiltedAlgebra} 
Let $H$ be a finite-dimensional tame hereditary $k$-algebra and $T\in \rmod H$ a tilting module, then the following statements are equivalent:
\begin{enumerate}
\item[(1)] $T$ contains a non-zero preprojective direct summand and a non-zero preinjective direct summand.
\item[(2)] $\End_H(T)$ is of finite representation type.
\end{enumerate}
\end{theorem}

However, we do not know yet whether the domestic Coxeter-Dynkin algebra can be obtained as the endomorphism algebra of a tilting \emph{module} over a tame hereditary algebra. Therefore, we have to take a closer look at the derived equivalence between them.

\begin{theorem}\label{thm:CDAlgebraRepFinite}
Let $B$ be a Coxeter-Dynkin algebra of canonical type, then B is representa-tion-finite if and only if it is domestic. Moreover, in this case B is a tilted algebra, meaning there exists a tilting module $T$ over a (tame) hereditary algebra $H$ such that $B\cong\End_H(T)$.
\end{theorem}
\begin{proof}
Note that one implication is clear according to the previous observations. Hence let $B$ be a domestic Coxeter-Dynkin algebra of canonical type. Combining Theorem \ref{thm:tiltSquidToOctopus} and \ref{thm:CanonicalTiltingModule} with \cite[Prop.~2.7]{Ringel}, we can deduce that there exists a tame hereditary algebra $H$ such that $\kD^b(\rmod H)\cong\kD^b(\rmod B)$ as triangulated categories. Recall that the module category of $H$ splits into three parts as
\begin{align*}
\rmod H = \kP \vee \kR \vee \kI,
\end{align*}
where $\kP$ is the preprojective component, $\kR$ are the regular components and $\kI$ is the preinjective component (see for instance \cite{DlabRingel}). Here, non-zero morphisms exist only from left to right. Since $H$ is hereditary, all indecomposable objects of the bounded derived category are stalk complexes. In order to prove the theorem we will show the existence of a tame hereditary algebra $H$ and a tilting complex\footnote{As the notation suggests, here we already associate the indecomposable summands of $T$ with indecomposable, projective $B$-modules resp. with vertices in the species \eqref{eq:speciesOctopus}.}
\begin{align*}
T=T_G~\oplus~ \bigoplus_{i=1}^t\bigoplus_{j=1}^{p_i-1}~ T_i(j)~ \oplus T_F \quad\in \kD^b(\rmod H)
\end{align*}
which satisfies the following properties:
\begin{enumerate}
\item[(i)] $T_G$ is the unique simple, projective $H$-module,
\item[(ii)] $T_i(j)$ lies in a regular component for all $1\leq i\leq t$ and $1\leq j\leq p_i-1$,
\item[(iii)] $T_F$ is a preinjective $H$-module.
\end{enumerate}
In particular, $T$ is even a tilting module. One of the main ingredients for this is that we have a \emph{rank function}
\begin{align*}
\rk=\rk_w:\quad \sK_0(\rmod H)\longrightarrow \ZZ,\qquad x\longmapsto \langle x,w \rangle
\end{align*}
which is a group homomorphism that depends on the choice of a generator $w$ of the radical of $\sK_0(\rmod H)$.
Up to a constant factor, this coincides with the rank function on a canonical bilinear lattice, introduced in \cite{LenzingKTheoreticStudy}. Note that one often also calls $\rk([X])$ the \emph{defect} of an $H$-module $X$ and that it provides information about whether $X$ is preprojective, regular or preinjective (see for instance \cite{DlabRingel}). For us, it is important that it is zero on classes of regular modules, while it changes the sign when passing from preprojective to preinjective modules. This carries over to the bounded derived category as depicted in Figure \ref{fig:derivedCategoryDomesticCase}.
\begin{figure}[ht]
\centering
\begin{tikzpicture}
\draw (0.5,0) -- (3,0);
\draw (1.5,0) -- (1,1.5);
\draw (1.25,2) -- (1,1.5);
\draw (0.5,2) -- (3,2);
\draw (4.5,-0.5) ellipse (1cm and 0.5cm);
\draw (4.5,2.5) ellipse (1cm and 0.5cm);
\draw (3.5,-0.5) -- (3.5,2.5);
\draw (5.5,-0.5) -- (5.5,2.5);
\draw (6,0) -- (9.5,0);
\draw (8,0) -- (7.5,1.5);
\draw (7.75,2) -- (7.5,1.5);
\draw (6,2) -- (9.5,2);
\draw (11,-0.5) ellipse (1cm and 0.5cm);
\draw (11,2.5) ellipse (1cm and 0.5cm);
\draw (10,-0.5) -- (10,2.5);
\draw (12,-0.5) -- (12,2.5);
\draw (12.5,0) -- (15,0);
\draw (14.5,0) -- (14,1.5);
\draw (14.25,2) -- (14,1.5);
\draw (12.5,2) -- (15,2);
\node at (4.7,-1.5) {$\underbrace{\hspace{6.4cm}}_{\rmod H}$};
\node at (11.2,-1.522) {$\underbrace{\hspace{6.4cm}}_{(\rmod H)[1]}$};
\node at (2.25,-0.7) {$\underbrace{\hspace{1.6cm}}_{\kP}$};
\node at (6.95,-0.7) {$\underbrace{\hspace{1.9cm}}_{\kI}$};
\node at (4.5,-0.8) {\footnotesize $\kR$};
\node at (1.75,3.5) {$\overbrace{\hspace{2.5cm}}^{\rk < 0}$};
\node at (4.5,3.5) {$\overbrace{\hspace{2cm}}^{\rk=0}$};
\node at (7.75,3.5) {$\overbrace{\hspace{3.5cm}}^{\rk > 0}$};
\node at (11,3.5) {$\overbrace{\hspace{2cm}}^{\rk=0}$};
\node at (13.75,3.5) {$\overbrace{\hspace{2.5cm}}^{\rk < 0}$};
\node at (1.2,1.45) {$\bullet$};
\node at (1.45,1.2) {\tiny $T_G$};
\node at (7.2,0.9) {$\bullet$};
\node at (7.3,0.65) {\tiny $T_F$};
\node at (3.7,0.5) {$\bullet$};
\node at (3.9,0.5) {$\bullet$};
\node at (3.7,0.85) {$\bullet$};
\node at (3.7,1.2) {$\bullet$};
\node at (3.9,1.2) {$\bullet$};
\node at (4.1,1.2) {$\bullet$};
\node at (4.3,1.2) {$\bullet$};
\node at (4.5,0.77) {\tiny $\bigoplus T_i(j)$};
\end{tikzpicture}
\caption{The category $\kD^b(H)\cong\kD^b(B)$ and the signs of the rank function applied to classes of indecomposable objects (the dots represent indecomposable direct summands of a tilting object with $\End(T)\cong B$).}
\label{fig:derivedCategoryDomesticCase}
\end{figure}
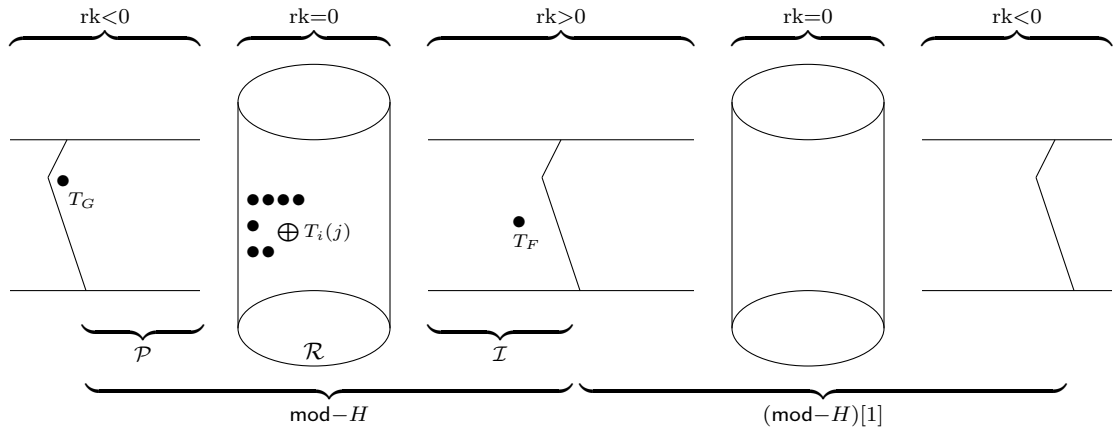
\\ We can now start with the construction of a tilting object with the desired properties.
\\ \underline{Step 1:}\quad Using Corollary \ref{cor:radical}, we can express the radical of $\sK_0(\rmod B)\cong\sK_0(\rmod H)$ in terms of the classes of simple $B$-modules. Hence, the rank function $\rk=\rk_w$ associated to the radical generator $w$ from Corollary \ref{cor:radical} satisfies
\begin{align*}
\rk([T_G])<0,\qquad \rk([T_i(j)])=0\quad\forall 1\leq i\leq t,~1\leq j\leq p_i-1,\qquad \rk([T_F])>0.
\end{align*}
In particular, the object $T_G$ lies in some (shifted) preprojective or preinjective component, but not in a regular component. By applying a suitable shift, we can assume that $T_G\in \kI[-1]\vee\kP$. Moreover, we can apply derived Auslander-Reiten translations to turn it into a projective $H$-module (sitting in degree 0). In order to achieve that $T_G$ is the unique simple, projective $H$-module we possibly have to modify the hereditary algebra $H$. More precisely, we can apply derived reflection functors and replace $H$ by a reflected tame hereditary algebra (in the sense of \cite[§2]{DlabRingel}). This proves that we can always find a tame hereditary algebra $H$ and a tilting complex $T$ such that condition (i) holds.\\
\\ \underline{Step 2:}\quad Let us fix some $1\leq i\leq t$ and $1\leq j\leq p_i-1$. Since $\rk([T_i(j)])=0$, we know that $T_i(j)$ lies in some (shifted) regular component. Hence, there exist a regular $H$-module $M\in\kR$ and an integer $n\in \ZZ$ such that $T_i(j)\cong M[n]$. Using projectivity of $T_G$ and
\begin{align*}
0\neq \Hom_B(e_GB,e_i(j)B)\cong\Hom_{\kD^b(H)}(T_G,T_i(j))\cong \Ext_H^n(T_G, M),
\end{align*}
it follows that $n=0$ and therefore $T_i(j)$ is concentrated in degree zero. This proves that property (ii) is satisfied.\\
\\ \underline{Step 3:}\quad Since $\rk([T_F])$ is positive, the object $T_F$ must lie in $\kI\vee\kP[1]$ (higher shifts are not possible, because $\Hom_{\kD^b(H)}(T_i(1),T_F)\neq 0$ and $H$ is hereditary). In order to verify that $T_F\in \kI$, we distinguish two cases:
\\ \underline{Case 1: $\varepsilon(d_1+\dots+d_t)\geq 3$} \quad In this case, the map $\theta_0$ from Proposition \ref{Prop:OctopusCondition} cannot be surjective which is why
\begin{align*}
0\neq W \cong \Hom_B(e_GB,e_FB)\cong\Hom_{\kD^b(H)}(T_G,T_F)
\end{align*}
and $T_F$ is concentrated in degree zero by the same argument as in step 2.
\\ \underline{Case 2: $\varepsilon(d_1+\dots+d_t)=2$} \quad This covers precisely the extended Dynkin types $\widetilde{A}_n$, $\widetilde{B}_n$, $\widetilde{C}_n$ and $\widetilde{BC}_n$. In these cases, since $T_G$ is the unique simple, projective $H$-module, every preprojective $H$-module $X$ satisfies $\Hom_H(T_G,X)\neq 0$. Hence, if $T_F$ was isomorphic to $X[1]$ for some preprojective $H$-module $X$, this would lead to 
\begin{align*}
0\neq\Hom_H(T_G,X)\cong\Hom_{\kD^b(H)}(T_G,T_F[-1])
\end{align*}
contradicting the fact that $T$ is a tilting object in $\kD^b(\rmod H)$. Therefore, $T_F$ is concentrated in degree 0 which proves condition (iii) also in the second case.\\
Summing up, we proved the existence of a tilting \emph{module} $T$ over $H$ such that $\End_H(T)\cong B$ and which has a preprojective and a preinjective direct summand. Hence $B$ is of finite representation type by Theorem \ref{thm:RingelRepFiniteTiltedAlgebra}.
\end{proof}

\begin{remark}
In the proof of the previous theorem we had to argue that the algebra $B$ (which was known to be derived equivalent to a hereditary algebra) is actually a \emph{tilted algebra}, meaning it arises as the endomorphism algebra of a tilting module over a hereditary algebra. We want to point out that this implication fails in general. An example of such an algebra would be the path algebra of the quiver
\begin{align*}
\xymatrix{
\bullet \ar[r]^a & \bullet \ar@<0.5ex>[r]^{b_1} \ar@<-0.5ex>[r]_{b_2} & \bullet \ar[r]^c & \bullet
}
\end{align*}
modulo the relations $ab_1=0=b_2c$. It is derived equivalent to a tame hereditary algebra of extended Dynkin type $\widetilde{A}_{3}$ (more precisely to a canonical algebra with weight sequence $(p_1,p_2)=(2,2)$). However, it cannot be realised as the endomorphism algebra of a tilting module over a tame hereditary algebra. Any such tilting object contains direct summands from different shifts of regular components. More examples and a detailed study of such \emph{quasi-tilted algebras} can be found for instance in \cite{LenzingSkowronskiQuasiTilted}.
\end{remark}

\begin{remark}\label{rmk:RepFiniteTiltedAlgebras}
Theorem \ref{thm:CDAlgebraRepFinite} also provides another perspective on Proposition \ref{Prop:OctopusCondition} and the question why we need an additional condition to be satisfied when defining the Coxeter-Dynkin algebra of canonical type. Suppose we have a domestic input datum, meaning a tame bimodule, weights and exceptional points such that $\delta<0$. In this case, we know that the squid algebra $A$ and the canonical algebra $C$ are derived equivalent to a tame hereditary algebra $H$. The condition of Proposition \ref{Prop:OctopusCondition} fails if and only if the oriented valued graph associated to $H$ is of the form 
\begin{align}\label{eq:valuedGraphAnTilde}
\begin{minipage}{12cm}
\xymatrix@R=1pc{
 & 2 \ar[r] & \dots & \dots \ar[r] & n \ar[dr] & \\
 1 \ar[ur] \ar[rrrrr] & & & & & n+1
}
\end{minipage}
\end{align}
where -- as usual -- the valuation $(1,1)$ at every edge is omitted in the diagram.
For a detailed study of such tame hereditary algebras we refer to \cite{DlabRingelTameHereditary}. If the base field $k$ is perfect, then $H$ is necessarily the tensor algebra of a $k$-modulation of \eqref{eq:valuedGraphAnTilde}. Otherwise it can also be of the form $\widetilde{A}_n(\varphi,\delta)$ (as defined in \cite{DlabRingelTameHereditary}).
In both cases, we claim that $\rmod H$ contains no tilting module with representation-finite endomorphism algebra. In regard of Theorem \ref{thm:RingelRepFiniteTiltedAlgebra}, this follows from the observation that an indecomposable, preprojective $H$-module $P$ and an indecomposable, preinjective $H$-module $I$ cannot be $\Ext$-orthogonal, meaning $P\oplus I$ is never rigid. First, by a similar argument as in the proof of Theorem \ref{thm:CDAlgebraRepFinite}, we can assume without loss of generality that $P$ is simple and projective and therefore has dimension vector $(0,\dots,0,1)$.
Since $\Hom_H(I,P)=0$, the combinatorial description of the Euler form yields
\begin{align*}
\dim\Ext_H^1(I,P)=-\langle[I],[P]\rangle
= \underbrace{-(\underline{\dim}~I)_{n+1}+(\underline{\dim}~I)_1}_{>0~~\text{since }I\text{ is preinjective}}+(\underline{\dim}~I)_n >0.
\end{align*} 
Summing up, we have shown that in the domestic case the condition from Proposition \ref{Prop:OctopusCondition} is equivalent to the existence of a representation-finite tilted algebra of the respective extended Dynkin type. Moreover, if the condition holds then the Coxeter-Dynkin algebra of canonical type is such a representation-finite tilted algebra.
\end{remark}

Finally, let us summarise some properties of canonical bilinear lattices, associated reflection groups and the signature of the \emph{symmetrised} Euler form as well as the representation type of the corresponding algebras:
\begin{align*}
\renewcommand{\arraystretch}{1.5} 
\begin{array}{|c|c|c|c|c|c|}
\hline
\begin{minipage}{1.7cm} \centering type \end{minipage}
& \begin{minipage}{0.7cm} \centering $\delta$ \end{minipage}
& \begin{minipage}{2cm} \vspace{4pt} \centering signature \\ $(+,0,-)$ \vspace{4pt} \end{minipage}
& \text{Weyl group}
& \begin{minipage}{2cm} \centering canonical algebra \end{minipage}
& \begin{minipage}{3cm} \centering Coxeter-Dynkin algebra \end{minipage}
\\
\hline\hline
\begin{minipage}{1.7cm} \centering domestic \end{minipage}
& \begin{minipage}{0.7cm} \centering $<0$ \end{minipage}
& (n-1,1,0)
& \begin{minipage}{2.5cm} \vspace{4pt}\centering affine \\ Coxeter group \vspace{4pt} \end{minipage}
& \begin{minipage}{2cm} \centering tame concealed \end{minipage}
& \begin{minipage}{3cm} \centering representation- \\ finite \end{minipage}
\\
\hline
\begin{minipage}{1.7cm} \centering tubular \end{minipage}
& \begin{minipage}{0.7cm} \centering $=0$ \end{minipage}
& (n-2,2,0)
& \begin{minipage}{2.5cm} \vspace{4pt} \centering elliptic \\ Weyl group \vspace{4pt} \end{minipage}
& \begin{minipage}{2cm} \centering tame \end{minipage}
& \begin{minipage}{3cm} \centering tame \end{minipage}
\\
\hline
\begin{minipage}{1.7cm} \centering wild \end{minipage}
& \begin{minipage}{0.7cm} \centering $>0$ \end{minipage}
& (n-2,1,1)
& \begin{minipage}{2.5cm} \vspace{4pt} \centering cuspidal \\ Weyl group \vspace{4pt} \end{minipage}
& \begin{minipage}{2cm} \centering wild \end{minipage} 
& \begin{minipage}{3cm} \centering wild \end{minipage}
\\
\hline
\end{array}
\end{align*}

Concerning the representation type, we recall that a finite-dimensional algebra is called \emph{tame concealed} if it arises as the endomorphism algebra of a preprojective tilting module over a tame hereditary algebra \cite[§4.3]{RingelTameAlgebras}. Canonical algebras of domestic type are tame concealed by \cite[§2.7]{Ringel}. We note that domestic squid algebras are not tame concealed since a corresponding tilting module over a tame hereditary algebra always contains a direct summand which is not preprojective but lies in a tube.


\end{document}